\documentclass[10pt]{article}
\usepackage{lmodern}
\usepackage{amsmath}
\usepackage[T1]{fontenc}
\usepackage[utf8]{inputenc}
\usepackage{authblk}
\usepackage{amsfonts}
\usepackage{graphicx}
\usepackage{rotating}
\usepackage{amssymb}
\usepackage[english]{babel}
\usepackage{color}
\usepackage{amsthm}
\usepackage{graphicx}
\usepackage{tkz-euclide}
\tikzset{elegant/.style={smooth,thick,samples=50,cyan}}
\usepackage{color,tikz}
\usetikzlibrary{patterns}
\usetikzlibrary{decorations.pathreplacing,calc}
\usepackage{rotating}
\usepackage{mathrsfs}
\usepackage{makecell}
\usepackage{microtype}
\usepackage{enumerate}
\usepackage{mathscinet}
\usepackage{array}
\usepackage{multirow}
\usepackage{booktabs}
\usepackage{enumerate}
\usepackage[cal=boondoxo,bb=ams]{mathalfa}
\usepackage{hyperref}
\hypersetup{hidelinks}
\usepackage{titlesec}
\usepackage{float}
\usepackage{tabularx}
\usepackage{booktabs}

\newtheorem{theorem}{Theorem}[section]
\newtheorem{proposition}{Proposition}[section]
\newtheorem{lemma}{Lemma}[section]

\newtheorem{remark}{Remark}[section]

\newcommand{\xii}{{\mathrm{|\xi|}}}

\newcommand{\supp}{{\mathrm{\,supp\,}}}

\renewcommand{\>}{\rangle}

\DeclareMathOperator{\lin}{lin}
\DeclareMathOperator{\nl}{non}

\title{Dispersive estimates for wave-type equations with time-dependent damping}
\author[1]{Halit Sevki Aslan\thanks{Halit S. Aslan (halitsevkiaslan@gmail.com)}}
\author[1]{Marcelo Rempel Ebert \thanks{Marcelo R. Ebert (ebert@ffclrp.usp.br) (corresponding author)}}
\affil[1]{Department of Computer Science and Mathematics (FFCLRP),
\newline University of S\~ao Paulo (USP), Ribeir\~ao Preto, SP 14040-901, Brazil}

\date{}
\numberwithin{equation}{section}
\allowdisplaybreaks

\newcommand{\R}{{\mathbb{R}}}
\newcommand{\Z}{{\mathbb{Z}}}

\newcommand{\low}{{\mathrm{low\,}}}

\newcommand{\BQ}{{\mathrm{BQ\,}}}
\begin{document}

\maketitle
\begin{abstract}
In this paper, we study the Cauchy problem for a class of semilinear evolution equations with scale-invariant time-dependent dissipation
\begin{equation*}
\begin{cases}
u_{tt} + L_{w^2}u + \dfrac{\mu}{1+t}u_t = \Delta^{\theta} f(u), & t>0,\ x\in\mathbb{R}^n,\\
u(0,x) = 0,\qquad u_t(0,x) = u_1(x), & x\in\mathbb{R}^n,
\end{cases}
\end{equation*}
where $\mu>0$, $f(u)=|u|^\alpha$ with $\alpha>1$, $\theta\in\{0,1\}$, and the operator $L_{w^2}$ is defined on the Fourier transform by multiplication by $w(\xi)^2$. We prove the global (in time) existence of small data solutions for $\alpha>\alpha_{\mathrm{crit}}$, where the critical exponent $\alpha_{\mathrm{crit}}$ depends on the choice of the operator $L_{w^2}$, the parameter $\mu$, and the nonlinear term. In particular, we consider two model cases. For Boussinesq-type operators with $w(\xi)=\sqrt{|\xi|^2+|\xi|^4}$, combined with the derivative-type nonlinearity $\Delta |u|^\alpha$, we obtain a Strauss-type critical exponent. On the other hand, for plate-type operators with $w(\xi)=|\xi|^\sigma$, $\sigma\geq2$, and power-type nonlinearity $|u|^\alpha$, the critical exponent is of Fujita type.
\medskip



\noindent\textbf{Keywords:} plate equation, Boussinesq equation, dispersive  estimates, global existence, Strauss exponent, Fujita exponent. \\
	
\noindent\textbf{AMS Classification (2020)}  35L15, 35L71, 35A01, 35B33.
	
\end{abstract}
\fontsize{12}{15}
\selectfont
\section{Introduction} \label{Sec:Introduction}
In this paper, we consider the following Cauchy problem for a semilinear evolution equations with scale-invariant time-dependent damping:
\begin{equation}\label{Eq:SemilinearProblem}
\begin{cases}
u_{tt} + L_{w^2}u + \dfrac{\mu}{1+t}u_t =\Delta^{\theta} f(u), & t>0,\ x\in\mathbb{R}^n,\\
u(0,x) = 0, \,\,\,\, u_t(0,x) = u_1(x), & x\in\mathbb{R}^n,
\end{cases}
\end{equation}
where $\mu> 0$,   $\theta\in \{0, 1\}$ and  the
 operator $L_{w^2}$ is defined via the Fourier transform by $ L_{w^2}u = \mathcal{F}^{-1}( w(\xi)^2\hat{u})$, where $\mathcal{F}$ denotes the Fourier transform on the space of tempered distributions $\mathcal{S}'$, and $\mathcal{F}^{-1}$ its inverse.
 Throughout the paper, the nonlinearity is regarded as a perturbation of the corresponding linear equation.
We assume that the nonlinearity $f$ satisfies
\begin{equation} \label{Eq:nonlinearity}
f(0)=0, \quad |f(u)-f(v)|\leq C\,|u-v|\big(|u|^{\alpha-1}+|v|^{\alpha-1}\big),
\end{equation}
for some $\alpha>1$. Assumption \eqref{Eq:nonlinearity} has the advantage of relying on Lipschitz-type regularity condition; several nonlinearities of classical interest, such as $|u|^\alpha$, $\pm u|u|^{\alpha-1}$, satisfy \eqref{Eq:nonlinearity}.
In order to treat the nonlinearity $\Delta|u|^\alpha$ in the framework of \eqref{Eq:nonlinearity}, we will consider initial data in homogeneous Sobolev spaces of negative order. Following \cite{Cho-Ozawa=2007} and the references therein, this assumption is commonly referred to as a ``vanishing condition''.


The Cauchy problem \eqref{Eq:SemilinearProblem} is often regarded as a \emph{critical model} in the sense that the asymptotic behavior of its solutions strongly depends on the size of the parameter $\mu$. Moreover, different choices of $w(\xi)$ lead to various well-known partial differential equations with scale-invariant damping, including the semilinear wave equation (Section \ref{Sect:wave-eq}), the plate equation (Section \ref{Sect:sigma-evol}), and the generalized Boussinesq equation (Section \ref{Sect:BQ}).

Wave equations with scale-invariant damping arise naturally in several problems from mathematical physics. In particular, such damping terms appear in compressible Euler equations with time-dependent dissipation, in the analysis of supersonic flows in divergent nozzles, and in generalized Tricomi-type models; see \cite{Li-Wang-Yin-2025} for a recent overview and further references. From the analytical viewpoint, the damping term $\frac{\mu}{1+t}u_t$ is critical in the sense that it separates diffusion-like behavior from predominantly hyperbolic dynamics. This critical structure is also reflected in the qualitative properties of solutions and in the determination of critical exponents for global existence and blow-up phenomena.

The main goal of this paper is to investigate the global (in time) existence of small data energy solutions to the semilinear problem \eqref{Eq:SemilinearProblem} under suitable assumptions on the initial data, the damping parameter $0<\mu<1$, the operator $L_{w^2}$, and the nonlinearity exponent $\alpha>1$.

Our analysis for the semilinear problem \eqref{Eq:SemilinearProblem} is based on $L^p-L^q$ decay estimates for the corresponding linear problem. In particular, for the plate-type model we employ $L^1-L^q$ estimates, and for the Boussinesq equation, by using dispersive estimates with the choice
\[ q=1+\alpha, \qquad p=q'=1+\frac{1}{\alpha}, \]
we are able to treat the nonlinear terms as a small perturbation of the associated linear problem, provided that the initial data is sufficiently small and the exponent $\alpha$ lies in the supercritical range.

 The asymptotic profile of solutions to \eqref{Eq:SemilinearProblem} strongly depends on the size of the damping parameter $\mu$. In the case $\mu>1$, the damping term produces stronger dissipative effects and the associated linear problem typically exhibits diffusion-like decay properties, which facilitate the treatment of the nonlinear terms. In contrast, the regime $0<\mu<1$ considered in this paper corresponds to a weakly damped hyperbolic behavior, where the dissipative mechanism is not sufficiently strong to dominate the dispersive effects. Consequently, the derivation of sharp decay estimates becomes substantially more delicate, especially for operators with non-homogeneous symbols such as those appearing in generalized Boussinesq-type equations. This weak damping regime therefore requires a finer analysis of the interplay between dispersion and dissipation.
\subsection*{The Strauss and the Fujita exponents}
The problem of determining the critical exponent for the semilinear wave equation
\[ u_{tt}-\Delta u = |u|^\alpha \]
goes back to \cite{John=1979}, who showed in dimension $n=3$ that global (in time) solutions with small initial data exist for $\alpha>1+\sqrt{2}$, while finite-time blow-up occurs for $\alpha\leq 1+\sqrt{2}$ under a suitable sign condition on the data. In \cite{Strauss=1981}, it was conjectured that the critical exponent for the wave equation is given by $\gamma(n-1)$, where $\gamma(n)$ is the positive root of the equation
\[ \frac{n}{2}\frac{\gamma-1}{\gamma+1}=\frac{1}{\gamma}. \]
This conjecture was later confirmed in \cite{GLS=1997} with the complete high-dimensional theory. Moreover, in \cite{Cho-Ozawa=2007} it was proved that $\gamma(n)$ is the critical exponent for the Boussinesq equation.  The shift from $\gamma(n)$ to $\gamma(n-1)$ is related to the structure of the phase function: the Hessian of $w(\xi)=|\xi|$ has rank $n-1$ for all $\xi\neq0$.

 On the other hand, the problem of determining the critical exponent for the semilinear damped wave equation
\[ u_{tt}-\Delta u + u_t = |u|^\alpha \]
was extensively studied in the literature. In particular, it was shown that the critical exponent coincides with the Fujita exponent $\alpha_{\text{Fuj}}=1+\frac{2}{n}$, that is well-known to be the critical exponent for the semilinear heat equation with power nonlinearity. More precisely, if $\alpha>\alpha_{\text{Fuj}}$, then global (in time) solutions with sufficiently small initial data exist, whereas for $\alpha\leq \alpha_{\text{Fuj}}$ solutions blow up in finite time under suitable assumptions on the initial data (see \cite{Todorova-Yordanov=2001} and \cite{Zhang=2001}). This reflects the diffusion phenomenon exhibited by the damped wave equation, whose asymptotic behavior is closely related to that of the heat equation.
\subsection{Wave equation with scale-invariant time-dependent damping} \label{Sect:wave-eq}
When $w(\xi) = |\xi|$, the operator $L_{w^2}$ reduces to the Laplacian, $L_{w^2} = -\Delta$, and \eqref{Eq:SemilinearProblem} becomes the semilinear wave equation with scale-invariant damping:
    \begin{equation}\label{Eq:Dabbicco=2014}
    \begin{cases}
    u_{tt} - \Delta u + \dfrac{\mu}{1+t} u_t = |u|^\alpha, & t \geq 0, \ x \in \mathbb{R}^n, \\
    u(0,x) = 0, \quad u_t(0,x) = u_1(x), & x \in \mathbb{R}^n,
    \end{cases}
    \end{equation}
where $\mu>0$ and $\alpha>1$.  This model has been extensively studied in the literature. It is conjectured that the small data Cauchy problem \eqref{Eq:Dabbicco=2014} with $u_1\in L^1(\mathbb{R}^n)$ admits a critical exponent of the form
\[ \alpha_{\text{crit}} = \max\left\{ \gamma(n+\mu)\,,\,\alpha_{\text{Fuj}}(n) \right\}, \qquad \text{that is}, \qquad \alpha_{\text{crit}} =
\begin{cases}
\gamma(n+\mu) & \text{if} \quad 0<\mu<\mu_*(n), \\
\alpha_{\text{Fuj}}(n) & \text{if} \quad \mu\geq \mu_*(n),
\end{cases} \]
where $\alpha_{\text{Fuj}}(n)=1+2/n$ and
\[ \mu_*(n) = \frac{n^2+n+2}{n+2} \]
satisfying $\gamma(n+\mu_*)=\alpha_{\text{Fuj}}(n)$. This conjecture has been partially verified in \cite{Dabbicco=2014, Dabbicco=2021, Dabbicco-Lucente=2013, Dabbicco-Lucente-Reissig=2015, Ikeda-Sobajima=2018}. \\
Recently,  in \cite{DA26} it was proved that $\alpha_{\text{crit}} = 1+\frac{4}{n}$ is the critical exponent to \eqref{Eq:Dabbicco=2014} under the assumption of small data  $u_1\in L^2(\mathbb{R}^n)$ and $\mu\geq 1$ if $n=1,2,3$ or $\mu\geq \frac{2n}{n+3}$ if $n\geq 4$.

Cauchy problems for the  semilinear   scale-invariant damping
wave models with time-dependent speed of propagation $u_{tt} - (1+t)^{2\ell}\Delta u + \dfrac{\mu}{1+t} u_t = |u|^\alpha$, with $\ell\in \mathbb{R}$, have been studied by many
 authors in a series of papers (see \cite{ AER23, GY25, TW22} and references therein).
  These models belong to the family of equations arising in Friedmann-Lema\^itre-Robertson-Walker (FLRW) spacetimes and are used to describe accelerating/decelerating expanding universe models.
\subsection{$\sigma$-evolution equation with scale-invariant time-dependent damping} \label{Sect:sigma-evol}
If $w(\xi) = |\xi|^\sigma$ with $\sigma > 0$, then $L_{w^2} = (-\Delta)^\sigma$, and \eqref{Eq:SemilinearProblem} becomes the following semilinear $\sigma$-evolution equation with scale-invariant time-dependent damping:
    \begin{equation}\label{Ex1:EbertMarquesNunes}
    \begin{cases}
    u_{tt} + (-\Delta)^\sigma u + \dfrac{\mu}{1+t} u_t = |u|^\alpha, & t \geq 0, \ x \in \mathbb{R}^n, \\
    u(0,x) = 0, \quad u_t(0,x) = u_1(x), & x \in \mathbb{R}^n,
    \end{cases}
    \end{equation}
where $\mu \geq 0$, $\sigma > 0$, and $\alpha > 1$. The choice $w(\xi) = |\xi|^2$ yields $L_{w^2} = \Delta^2$ and this describes the dynamics of a plate with a time-dependent damping mechanism.

For $\sigma>1$ and $\mu=0$, assuming small initial data in $L^1(\mathbb{R}^n)\cap L^2(\mathbb{R}^n)$, the authors of \cite{Ebert-L} proved that, for $\sigma<n\leq 2\sigma$, the exponent $\alpha=\frac{n+\sigma}{(n-\sigma)_+}$ is the critical exponent distinguishing blow-up and global existence for \eqref{Ex1:EbertMarquesNunes}.

The model \eqref{Ex1:EbertMarquesNunes} with $\sigma>1$ has been recently investigated in \cite{Ebert-Marques-Nunes=2024}, where some global existence results for small data solutions were obtained.  In particular, the authors proved that
\begin{itemize}
\item if $1\leq n<2\sigma$ and  $\mu>\max\{\frac{4n}{n+2\sigma}\,,\,1 \}$, then the Fujita-type exponent $p=1+\frac{2\sigma}{n}$ is the critical exponent distinguishing blow-up and global existence by assuming small initial data in $L^1(\mathbb{R}^n)\cap L^r(\mathbb{R}^n)$ where $r\in[1,2]$;
\item if $1\leq n<\sigma$ and $1-n/\sigma<\mu<\min\{1\,,\,\mu^{\star}\}$, then the critical exponent is a shift of Fujita exponent $\alpha=\alpha_{\text{Fuj}}(\mu+n/\sigma-1)$ by assuming small initial data in $L^1(\mathbb{R}^n)\cap L^2(\mathbb{R}^n)$. Here $\mu^{\star}$ is defined by $\mu^\star = \frac{1}{2}\big( 1-n/\sigma +\sqrt{(1-n/\sigma)^2+8(1-n/\sigma)} \big)$.
\end{itemize}
\subsection{Boussinesq equation with scale-invariant time-dependent damping} \label{Sect:BQ}
Choosing $w(\xi) = \sqrt{|\xi|^{2} + |\xi|^{4}}$, leads to $L_{w^2} = -\Delta + \Delta^2$, where the low frequency and high frequency behaviors are governed by different Laplacians. Namely, we consider
\begin{equation}\label{Eq:Boussinesq}
\begin{cases}
u_{tt} - \Delta u + \Delta^2 u + \dfrac{\mu}{1+t} u_t = \Delta f(u) , & t\geq0, \ x\in \mathbb{R}^n,  \\
u(0,x)=0, \quad u_t(0,x)=u_1(x), & x\in \mathbb{R}^n,
\end{cases}
\end{equation}
where $0<\mu<1$ and $f(u)=|u|^\alpha$ with $\alpha>1$. For discussions on the physical background and modeling aspects of the damped Boussinesq equations, we refer the reader to \cite{Wang-Luo-Li=2023}, where the authors studied the following Cauchy problem for the generalized Boussinesq equation with weak damping:
\begin{equation}\label{Boussinesq-with-damping}
\begin{cases}
u_{tt}-\Delta u+\Delta^2 u + u_t = \Delta f(u) , & t\geq0, \ x\in \mathbb{R}^n,  \\
u(0,x)=u_0(x), \quad u_t(0,x)=u_1(x), & x\in \mathbb{R}^n,
\end{cases}
\end{equation}
where the nonlinear term $f(u)$ behaves as a power type satisfying $f(u)=O(|u|^2)$ as $u\to 0$. They established the global existence and time-decay rates of solutions to the Cauchy problem \eqref{Boussinesq-with-damping} by applying the energy method together with suitable interpolation inequalities. Moreover, the existence, uniqueness, and asymptotic stability of time-periodic solutions to the equation in \eqref{Boussinesq-with-damping} were investigated in \cite{Wang-Li=2018}.

Although the coefficient in front of the second-order operator is constant in \eqref{Eq:Boussinesq}, time-dependent coefficients naturally arise in several related beam and plate models. For instance, the authors in \cite{DAE-ASY=2021} investigated damped plate equations with a time-dependent coefficient of the form $u_{tt}-\lambda(t)\Delta u+\Delta^2 u + u_t =0$, while, in \cite{Yoshikawa-Wakasugi=2018}, the authors studied the asymptotic behavior of solutions to damped beam equations with variable coefficients, which may be regarded as the case $n=1$ of the previous model. More recently, the asymptotic behavior of solutions to a nonlinear beam equations with two time-dependent coefficients, $u_{tt}-\lambda(t)\partial_x^2 u+\partial_x^4 u + b(t) u_t =\partial_x f(\partial_x u)$, was studied  in \cite{HWY=2024}. Motivated by these developments, it is natural to investigate related higher-order models with time-dependent damping. Nevertheless, to the best of our knowledge, the Cauchy problem \eqref{Eq:Boussinesq} with scale-invariant damping has not yet been addressed in the literature.

On the other hand, the undamped Boussinesq equation (that is, when $\mu\equiv0$) have drawn much attention by many researchers due to the wide applications in the real world (see \cite{Bona-Sachs=1988, Cho-Ozawa=2007, Farah=2009A, Linares1993, Linares1995, Liu=1997}). The classical form
\[ u_{tt}-u_{xx}+u_{xxxx} = (u^2)_{xx}, \]
was first derived by Boussinesq in 1872 in \cite{Boussinesq1872} to describe the propagation of long waves with small amplitude on the surface of shallow water, where $u=u(x,t)$ denotes the elevation of the free fluid surface. Beyond water waves, the Boussinesq equation also models nonlinear vibrations of strings and electromagnetic waves in nonlinear dielectric materials (see \cite{Turitsyn1993, Zakharov1973}), among other physical phenomena.
\medskip

\noindent\textbf{Notation:}
\begin{itemize}
\item For any multi-index $\gamma=(\gamma_1,\ldots,\gamma_n)\in\mathbb{N}^n$, we use the notation
\[ \partial_x^\gamma = \partial_{x_1}^{\gamma_1}\cdots\partial_{x_n}^{\gamma_n}, \qquad \xi^\gamma=\xi_1^{\gamma_1}\cdots \xi_n^{\gamma_n}, \qquad |\gamma|=\gamma_1+\cdots +\gamma_n. \]
\item Let $\hat f=\mathcal{F}f$ denote the Fourier transform, with respect to the spatial variable $x$, of a tempered distribution $f$. If $f\in L^1(\mathbb{R}^n)$, then
\[ \hat f(\xi) = \int_{\mathbb{R}^n} e^{-ix\cdot\xi}f(x)\,dx, \]
where $x\cdot\xi=\sum_{j=1}^n x_j\xi_j$. The inverse Fourier transform is given by $\mathcal{F}^{-1}f(x) = (2\pi)^{-n}\mathcal{F}f(-x)$.
\item We write $(a)_+:=\max\{a\,,\,0\}$ and, as usual, $1/(a)_+:=+\infty$ if $a\leqslant0$.
\item For $p\in[1,\infty]$, $p'$ denotes its conjugate exponent, i.e., $\frac1p+\frac1{p'}=1$.
\item $L_p^q = L_p^q(\mathbb{R}^n)$ denotes the space of tempered distributions $T\in \mathcal{S}'(\R^n)$ such that $T\ast f\in L^q(\R^n)$ for any $f\in \mathcal{S}(\R^n)$, where $\mathcal{S}(\R^n)$ stands for the space of rapidly decreasing Schwartz functions, and
\[ \|T\ast f\|_{L^q} \leq C\|f\|_{L^p} \]
for all $f\in \mathcal{S}(\R^n)$ with a constant $C>0$, which is independent of $f$. In this case, the operator $T\ast$ is extended by density to the space $L^p(\R^n)$.

\item $M_p^q = M_p^q(\mathbb{R}^n)$ stands for the set of Fourier transforms $\hat{T}$ of distributions $T\in L_p^q$, equipped by the norm
\[ \|m\|_{M_p^q} := \sup\big\{ \big\| \mathcal{F}^{-1}\big( m\mathcal{F}(f) \big) \big\|_{L^q}: f\in\mathcal{S}(\R^n), \|f\|_{L^p} = 1 \big\}, \]
and we set $M_p = M_p^p$. The elements in $ M_p^q$ are called multipliers of type $(p,q)$.

\item The pseudo-differential operators $|D|$ and $\langle D\rangle$ are defined through their symbols $|\xi|$ and $\langle\xi\rangle$, respectively, where $
\langle\xi\rangle := \sqrt{1+|\xi|^2}$ denotes the Japanese bracket.
\end{itemize}
%
\section{Main results} \label{Sec:Main-Results}
Our first goal is to fill some gaps in the global existence results for small data solutions to \eqref{Ex1:EbertMarquesNunes} obtained in the previous works \cite{Ebert-L} and \cite{Ebert-Marques-Nunes=2024}. Applying the dispersive estimates derived in Proposition \ref{Prop1H}, we obtain the following result.
\begin{theorem}\label{thm:Nonlinear-plate-1}
Let $\sigma\geq 2$, $1\leq n<2\sigma$ and $\big(1-\frac{n}{\sigma}\big)_+<\mu<\min\big\{ 1, 2-\frac{n}{\sigma} \big\}$. Assume that
\[ 1+\frac{2}{\mu+\dfrac{n}{\sigma}-1}<\alpha. \]
Then, for sufficiently small initial data
\[ u_1\in \mathcal A:=L^1(\mathbb R^n)\cap L^2(\mathbb R^n) \]
there is a unique global (in time) solution $u\in \mathcal{C}\big([0,\infty);L^2(\mathbb R^n)\cap L^r(\mathbb R^n)\big)$ for all $r\in (\alpha,\infty)$ to \eqref{Ex1:EbertMarquesNunes}.
Moreover, the solution satisfies the following decay estimates for all $t\geq 0$ and $q\in[r,\infty)$:
\begin{equation}\label{eq:decayNP-1}
\|u(t,\cdot)\|_{L^q} \lesssim (1+t)^{1-\mu-\frac{n}{\sigma}\left(\frac{1}{1+\delta}-\frac{1}{q} \right)} \|u_1\|_{\mathcal A},
\end{equation}
 for any $\delta>0$ such that $(1+\delta)\alpha<r$.
\end{theorem}
\begin{remark}
Theorem \ref{thm:Nonlinear-plate-1} improves the global existence result stated in Theorem 2.1 in \cite{Ebert-Marques-Nunes=2024} in several aspects. First, the assumption on the space dimension is relaxed from $n<\sigma$ to $n< 2\sigma$. Second, the upper bound for damping parameter is simplified. In \cite{Ebert-Marques-Nunes=2024}, the parameter $\mu$ is required to satisfy
\[ 1-\frac{n}{\sigma}<\mu<\min\{1,\mu^\star\}, \]
where $\mu^\star$ is defined at the end of Section \ref{Sect:sigma-evol}.

Moreover, the restriction on the exponent of the nonlinearity is replaced by
\[ 1+\frac{2}{\mu+\frac{n}{\sigma}-1}<\alpha, \]
instead of the interval
\[ 1+\frac{2}{\mu+\frac{n}{\sigma}-1}<\alpha \leq \frac{2\gamma_1}{(2\gamma_1+\mu-2)_+}, \quad \gamma_1=\frac{n}{\sigma} \]
appearing in \cite{Ebert-Marques-Nunes=2024}. Namely, we have a result without imposing an upper bound on the power of the nonlinearity in Theorem 2.1.

Theorem \ref{thm:Nonlinear-plate-1} yields a simpler description of the decay behavior together with a broader range of admissible parameters. Finally, duo to the lack of $L^1-L^q$ estimates in Proposition 3.2, we have the loss  of decay  $(1+t)^{\frac{n\delta}{\sigma(1+\delta)}}$  in \eqref{eq:decayNP-1}  for solutions to \eqref{Ex1:EbertMarquesNunes}.
\end{remark}

Our second goal is to derive the Strauss-type exponent to \eqref{Eq:Boussinesq}.
\begin{theorem}\label{thm:NBQ}
Let $\mu\in (0,1)$. Assume that
\[ \alpha_{\text{crit}}<\alpha<1+\frac{4}{(n-2)_+}, \]
where $\alpha_{\text{crit}}$ is the positive solution to
\begin{equation}\label{eq:StraussBoussinesq}
\frac{n}2\frac{\alpha-1}{\alpha+1}+\frac{\mu}{2}=\frac1\alpha\Big( 1+\frac{\mu}{2} \Big).
\end{equation}
If $n\geq 2$, then for sufficiently small initial data
\[ u_1\in\mathcal{A}:= \dot H^{-1, 1+\frac1\alpha}(\mathbb{R}^n)\cap L^{1+\frac1\alpha}(\mathbb{R}^n),\]
there is a unique global (in time)  solution $u\in L^\infty\big( [0, \infty), L^{\alpha+1}(\mathbb{R}^n) \big)$ to \eqref{Eq:Boussinesq}.

Moreover, the solution satisfies the following decay estimates:
\begin{equation}\label{eq:decayNIBQ}
\|\, u(t,\cdot)\|_{L^{\alpha+1}} \lesssim (1+t)^{-\delta_\BQ-\frac{\mu}2}\|u_1\|_{\mathcal{A}},\quad \delta_\BQ=\frac{n}2 \frac{\alpha-1}{\alpha+1}.
\end{equation}
In space dimension $n=1$, the result holds replacing $\mathcal{A}$ by
\[\mathcal{A}:= \dot H^{-\frac32, 1+\frac1\alpha}(\mathbb{R}^n)\cap H^{-1, 1+\frac1\alpha}(\mathbb{R}^n).\]
\end{theorem}
\begin{remark}
In the case $\mu=0$ in \eqref{eq:StraussBoussinesq}, we find the classical Strauss exponent $\gamma(n)$ and our result corresponds to the result obtained in~\cite[Theorem 2]{Cho-Ozawa=2007}.
\end{remark}
\begin{remark}
In the present paper, we investigated the weakly damped regime $0<\mu<1$ in \eqref{Eq:Boussinesq}, where the asymptotic behavior of solutions remains essentially hyperbolic. In this case, the competition between dispersion and weak dissipation leads to a Strauss-type critical exponent for the Boussinesq equation in Theorem \ref{thm:NBQ}.

On the other hand, for larger damping parameters $\mu>1$ in \eqref{Eq:Boussinesq}, one may expect a qualitatively different behavior. Indeed, the damping term is sufficiently strong to produce diffusion-like effects in the low-frequency region, similarly to what occurs for damped wave and beam equations with effective dissipation. In view of this phenomenon, it is natural to conjecture that the critical exponent associated with \eqref{Eq:Boussinesq} for $\mu>1$ and the nonlinearity $|u|^\alpha$ with $\alpha>1$ should be of Fujita type. This expectation is also consistent with the decay structure observed for weakly damped beam equations in \cite{Ikehata-Soga=2015, Takeda-Yoshikawa=2012, Takeda-Yoshikawa=2013}, where the low-frequency part exhibits heat-like behavior and stronger dissipative effects dominate the long-time behavior.
\end{remark}
\section{Dispersive estimates}
The corresponding $s-$parameter linear model to \eqref{Eq:SemilinearProblem} with vanishing right-hand side is given by
\begin{equation}\label{Eq:LinearProblem}
\begin{cases}
u_{tt} + L_{w^2}u + \dfrac{\mu}{1+t}u_t = 0, & t\geq s\geq 0,\,\,\, x\in\mathbb{R}^n,\\
u(s,x) = 0, \,\,\,\, u_t(s,x) = u_1(s,x), & x\in\mathbb{R}^n.
\end{cases}
\end{equation}
We apply the partial Fourier transformation $\hat{u}(t,\xi)=\mathcal{F}_{x\to\xi}(u)(t,\xi)$ with respect to spatial variables and obtain
\begin{equation}\label{Eq:LinearProblem-Fourier}
\begin{cases}
\widehat{u}_{tt} + w(\xi)^2\widehat{u} + \dfrac{\mu}{1+t}\widehat{u}_t = 0 & t\geq s\geq 0,\,\,\, \xi\in\mathbb{R}^n,\\
\widehat{u}(s,\xi) = 0, \,\,\,\, \widehat{u}_t(s,\xi) = \widehat{u}_1(s,\xi), &  \xi\in\mathbb{R}^n.
\end{cases}
\end{equation}
Introducing the new variable $\tau:=(1+t)w(\xi)$ and making the ansatz $\widehat{u}(t,\xi)=v(\tau)$, the Cauchy problem \eqref{Eq:LinearProblem-Fourier} reduces to the following Cauchy problem for second order ODE:
\begin{equation} \label{Eq:Bessel-type}
\begin{cases}
v_{\tau\tau} + v + \dfrac{\mu}{\tau}v_\tau = 0, & t\geq s\geq 0,\,\,\, \xi\in\mathbb{R}^n,\\
v((1+s)w(\xi)) = 0, \,\,\,\, v_\tau((1+s)w(\xi)) = \dfrac{\hat{u}_1(s,\xi)}{w(\xi)}, & \xi\in\mathbb{R}^n.
\end{cases}
\end{equation}
According to \cite{Wirth=2004}, the solution to the transformed equation \eqref{Eq:Bessel-type} can be expressed in terms of Hankel functions $H^{\pm}_{\rho}$. This representation captures the oscillatory and decaying behavior of solutions at both low and high frequencies and is crucial for deriving sharp pointwise estimates in the Fourier space (see Lemma \ref{hankel}).
\begin{proposition} \label{Prop-Hankel}
Assume that $u$ solves the Cauchy problem \eqref{Eq:LinearProblem}. Then, the Fourier transform $\hat{u}(t, s,\xi)$ can be represented as
\[ \widehat{u}(t,s,\xi)= \psi(t,s,\xi) \widehat{u}_1(s,\xi), \]
where the multiplier $\psi=\psi(t,s,\xi)$ satisfies
\begin{equation*}\label{Eq:Hankel}
\psi(t,s,\xi)=\frac{\pi}{4i}\frac{(1+t)^{\rho}}{(1+s)^{\rho-1}}
\left|
\begin{array}{cc}
H^-_{\rho}\left( (1+s)w(\xi) \right) & H^-_{\rho}\left( (1+t)w(\xi) \right) \\
H^+_{\rho}\left( (1+s)w(\xi) \right) & H^+_{\rho}\left((1+t)w(\xi) \right) \\
\end{array} \right|,
\end{equation*}
with
\[ \rho= \frac{1-\mu}{2}. \]
\end{proposition}
Let us introduce $m=m(t,s,\xi)$ as follows:
\[ m(t,s,\xi) :=
\left|
\begin{array}{cc}
H^-_{\rho}\left( (1+s)w(\xi) \right) & H^-_{\rho}\left( (1+t)w(\xi) \right) \\
H^+_{\rho}\left( (1+s)w(\xi) \right) & H^+_{\rho}\left((1+t)w(\xi) \right) \\
\end{array} \right|. \]
Moreover, we recall that for $\rho\notin \Z$
\begin{equation} \label{Eq:Hankel-property}
H_{\rho}^{+}(\tau)= (i\sin(\rho\pi))^{-1}\big( J_{-\rho}(\tau)-J_{\rho}(\tau)e^{-i\rho \pi} \big) \quad \text{and} \quad  H_{\rho}^{-}(\tau)= (i\sin(\rho\pi))^{-1}\big(e^{i\rho \pi}  J_{\rho}(\tau)-J_{-\rho}(\tau) \big).
\end{equation}

We divide the extended phase space $[0,\infty)\times\mathbb{R}^n_\xi$ into following zones:
\begin{align*}
Z_{\text{high}} &= \{ (t,\xi)\in[0,\infty)\times\mathbb{R}^n_\xi : |\xi|\geq 1 \}, \\
Z_1 &= \big\{ (t,\xi)\in[0,\infty)\times\mathbb{R}^n_\xi : 1\leq (1+s)w(\xi) \big\}\cap\{ (t,\xi)\in[0,\infty)\times\mathbb{R}^n_\xi : |\xi|\leq 1 \}, \\
Z_2 &= \big\{ (t,\xi)\in[0,\infty)\times\mathbb{R}^n_\xi : (1+s)w(\xi)\leq 1\leq (1+t)w(\xi)\big\},\\
Z_3 &= \big\{ (t,\xi)\in[0,\infty)\times\mathbb{R}^n_\xi : (1+t)w(\xi)\leq 1 \big\}.
\end{align*}
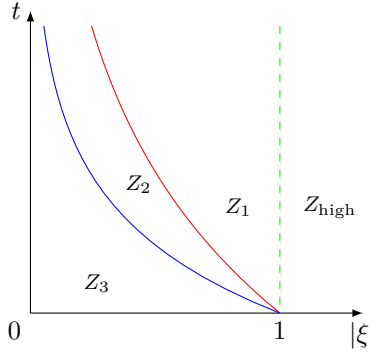
\begin{figure}[H]
\begin{center}
\begin{tikzpicture}[>=latex,xscale=1.1]
	\draw[->] (0,0) -- (4,0)node[below]{$|\xi|$};
	\draw[->] (0,0) -- (0,4)node[left]{$t$};
    \node[below left] at(0,0){$0$};
	\draw[domain=0:3.8,color=red,variable=\t] plot ({3*exp(-\t/2.7)},\t);
	\node[color=black] at (2.5, 1.4){{\footnotesize $Z_1$}};
    \node[color=black] at (3.6, 1.4){{\footnotesize $Z_{\text{high}}$}};
	\node[color=black] at (1.3,1.7){{\footnotesize $Z_2$}};
	\draw[domain=0:3.8,color=blue,variable=\t] plot ({3*exp(-\t/1.3)},\t);
    \node[below] at(3,0){$1$};
    \draw[dashed, color=green] (3, 3.8) -- (3, 0);
	\node[color=black] at (.8,.4){{\footnotesize $Z_3$}};
\end{tikzpicture}
\caption{\small Division of extended phase space into zones}
\label{fig.zone}
\end{center}
\end{figure}
More precisely, we will consider a bump function $\chi\in\mathcal{C}_{0}^{\infty}([0, \infty))$ such that $\chi\equiv 1$ on [0,1], $\supp \chi \subset[0,2]$ and $\chi^{\prime} \leq 0$. Then, the functions $\chi_{1}, \chi_{2}$ and $\chi_{3}$ are defined by
\begin{align*}
\chi_{1}(s, \xi) &= 1-\chi\left((1+s)w(\xi)\right), \\
\chi_{2}(t, s, \xi) &= \chi\left((1+s)w(\xi)\right)\left(1-\chi\left((1+t)w(\xi)\right)\right),\\
\chi_{3}(t, s, \xi) &= \chi\left((1+s)w(\xi)\right) \chi\left((1+t)w(\xi)\right),
\end{align*}
and they satisfy the condition $\chi_{1}+\chi_{2}+\chi_{3} \equiv 1$ on the whole extended phase space.

To obtain long time decay estimates for the solution to \eqref{Eq:LinearProblem}, we will use Mikhlin-H\"{o}rmander multiplier theorem in the following form.
\begin{lemma} [Mikhlin-H\"{o}rmander multiplier theorem, Theorem 2.5 in \cite{Hormander=1960}] \label{Thm:Mikhlin-Hormander}
Let $1<p<\infty$ and $k=[n/2]+1$. Assume that $h\in\mathcal{C}^k(\mathbb{R}^n\textbackslash\{0\})$ and
\[ |\partial_\xi^\gamma h(\xi)| \leq C|\xi|^{-|\gamma|}, \qquad |\gamma|\leq k. \]
Then, $h\in M_p^p$.
\end{lemma}
Let $g\in\mathcal{C}^\infty(\mathbb{R}^n\textbackslash\{0\})$ and
\begin{equation}\label{eq:gH}
|\partial_\xi^\gamma g(\xi)|\leq C\,|\xi|^{\eta-|\gamma|},
\end{equation}
for some $\eta\in\R$ (in particular, \eqref{eq:gH} holds if $g$ is homogeneous of degree $\eta$).
As a consequence of Lemma \ref{Thm:Mikhlin-Hormander} we have the following.
\begin{lemma}\label{lemma:m1-for-Plate-BQ}
Let $a\in S^{-\frac12}(\mathbb{R}^n)$, that is, $a\in \mathcal{C}^\infty(\mathbb{R}^n)$ and satisfy
\[ |\partial_\xi^\gamma a(\xi)|\leq C\,\langle \xi\rangle^{-\frac12-|\gamma|}, \quad \gamma\in \mathbb{N}^n. \]
Let us define
\[ h(t, \xi):=((1+t)g(\xi))^{\frac12}a((1+t)g(\xi)), \qquad t\ge0,  \]
with $g\in\mathcal{C}^\infty(\mathbb{R}^n\textbackslash\{0\})$ and satisfy \eqref{eq:gH}. Then $h(t, \cdot)\in M_p^p$ for every $p\in(1,\infty)$, and there exists a constant $C>0$, independent of $t\ge0$, such that $\|h(t, \cdot)\|_{M_p^p}\leq C$.
\end{lemma}
\begin{proof}
We prove that $h$ satisfies the Mikhlin--H\"ormander condition uniformly with respect to $t\ge0$, namely,
\[ |\partial_\xi^\alpha h(t, \xi)|\le C_\alpha |\xi|^{-|\alpha|}, \qquad \alpha\in\mathbb N^n. \]
Hence, the multiplier estimate follows.

If the   multiplier $g$ is homogeneous, then
$\|h(t, \cdot)\|_{M_p^p}=\|h(0, \cdot)\|_{M_p^p}$ and we may reduce the problem to time-independent multiplier.

In the following, our analysis include also  non-homogeneous multipliers. Set
\[ F(r):=r^{\frac12}a(r), \qquad r>0. \]
Since $a\in S^{-\frac12}$, one has
\[ |F^{(k)}(r)|\le C_k r^{-k}, \qquad k\in\mathbb N_0. \]
Moreover,
\[ h(t, \xi)=F((1+t)g(\xi)). \]
By the chain rule, for every multi-index $\alpha\neq0$,
\[ \partial_\xi^\alpha h(t, \xi) = \sum_{k=1}^{|\alpha|} F^{(k)}((1+t)g(\xi)) (1+t)^k
\sum_{\substack{\beta_1+\cdots+\beta_k=\alpha\\ |\beta_j|\ge1}}
C_{\alpha,\beta} \prod_{j=1}^k \partial_\xi^{\beta_j}g(\xi). \]
Since
\[ |\partial_\xi^\gamma g(\xi)|\le C_\gamma |\xi|^{\eta-|\gamma|}, \]
we obtain
\begin{align*}
|\partial_\xi^\alpha h(t, \xi)|
&\leq C \sum_{k=1}^{|\alpha|} ((1+t)g(\xi))^{-k}(1+t)^k |\xi|^{k\eta-|\alpha|} \leq C_\alpha |\xi|^{-|\alpha|}.
\end{align*}
Hence, $ h(t, \xi)$ satisfies the Mikhlin--H\"ormander condition, and therefore $\| h(t, \cdot)\|_{M_p^p}\leq C$, $p\in(1,\infty)$ with $C$ independent of $t$. 
\end{proof}

\begin{remark}
The conclusion of Lemma \ref{lemma:m1-for-Plate-BQ} also holds in the case that the constant $\eta$ change at low and high frequencies (see  for instance \cite{M80}), namely, if $g$ is smooth on $\mathbb R^n\setminus\{0\}$, and satisfies
\begin{align*}
|\partial_\xi^\gamma g(\xi)|\le C_\gamma
\begin{cases}
|\xi|^{\eta_1-|\gamma|}, & |\xi|\leq 1,\\
|\xi|^{\eta_2-|\gamma|}, & |\xi|\geq 1.
\end{cases}
\end{align*}
\end{remark}

To quantify the influence of dispersion on $L^p$–$L^q$ estimates for $1\leq p\leq 2\leq q\leq \infty$, we define
\begin{equation}\label{eq:dH}
d=d(p,q):=\min\Big\{ \frac{2}{p}-1,\;1-\frac{2}{q} \Big\}.
\end{equation}
In particular, one has $d(p,q)\leq \frac{1}{p}-\frac{1}{q}$, with equality if and only if $q=p/(p-1)$, that is, when $(p,q)$ lies on the conjugate line. Moreover, let
\begin{equation}\label{eq:kappaH}
\kappa=n\Big(\frac1p-\frac1q\Big)+\eta.
\end{equation}
\subsection{Plate-type equation}\label{sec:linear-plate}
In this section, we consider the following Cauchy problem for the linear equation:
\begin{equation}\label{eq:Homogeneous}
\begin{cases}
u_{tt}+ (-\Delta)^\sigma u+ \dfrac{\mu}{1+t} u_t =0, & t\geq s\geq 0,\,\,\, x\in\mathbb{R}^n,\\
u(s,x)=0, \quad u_t(s,x)=u_1(s, x), & x\in\mathbb{R}^n.
\end{cases}
\end{equation}
Here, $\sigma>0$, $\sigma\neq1$, and $(-\Delta)^\sigma$ denotes the fractional Laplacian defined via the Fourier transform by
\[ \mathcal{F}\left((-\Delta)^\sigma u\right)(\xi) = |\xi|^{2\sigma}\hat{u}(\xi). \]
\begin{remark}[General framework]
More generally (see for instance \cite{DEL=2026}), one may consider operators $L_{w^2}$ with symbol $w(\xi)^2$, where the phase function $w$ is real-valued, positive for $\xi\neq0$, homogeneous of degree $\sigma>0$, $\sigma\neq1$, and sufficiently smooth on $\mathbb{R}^n\setminus\{0\}$ (for instance, $w\in\mathcal{C}^{n+3}(\mathbb{R}^n\setminus\{0\})$).

In this general setting, we define
\[ \mathrm{r}:=\min_{\xi\in S^{n-1}}\operatorname{rank} H_w(\xi), \]
where $H_w(\xi)$ denotes the Hessian matrix of $w$. If $w$ is homogeneous of degree $\sigma$, then:
\begin{itemize}
\item $\mathrm{r}\leq n-1$ if $\sigma=1$,
\item $\mathrm{r}=n$ may occur if $\sigma\neq1$.
\end{itemize}
In particular, for the fractional Laplacian $w(\xi)=|\xi|^\sigma$, corresponding to $L_{w^2}=(-\Delta)^\sigma$, one has
\[
\mathrm{r}=
\begin{cases}
n, & \text{if}\,\,\, \sigma\neq1,\\
n-1, & \text{if}\,\,\, \sigma=1.
\end{cases}
\]
\end{remark}
%
We preliminarily notice that, thanks to the homogeneity assumption of $w(\xi) =\xii^\sigma$, letting $\xi'=\xi/\xii$, we have the following:
\begin{align*}
|\partial_\xi^\gamma w(\xi)| & = \xii^{\sigma-|\gamma|} |\partial_\xi^\gamma w(\xi')| \leq C_\gamma\,\xii^{\sigma-|\gamma|},\\
w(\xi) & =\xii^\sigma\,w(\xi')\geq c\xii^\sigma, \qquad c=\min_{|\xi|=1} w(\xi)>0.
\end{align*}
Moreover, letting $\eta=(1+s)^{\frac1{\sigma}}\xi$, by scaling property of  multipliers we have
\begin{equation}\label{simple}
\| m(t, s, \cdot) \|_{M_p^q} \leq (1+s)^{-\frac{n}{\sigma}\left(\frac1p-\frac1q\right)} \big\| m\left(\frac{t-s}{1+s}, 0, \cdot\right) \big\|_{M_p^q}.
\end{equation}
Hence, we may simplify our analysis in the extended phase space by splitting the study of the multiplier into low- and high-frequency regimes. Moreover, in the low-frequency regime, we further decompose the analysis as follows:
\begin{equation}\label{simple2}
\chi(w(\xi))m\left(\frac{t-s}{1+s}, 0, \cdot\right) = \chi(w(\xi))\left( 1-\chi\big(\frac{1+t}{1+s}|\xi|^{\sigma}\big)+ \chi\big(\frac{1+t}{1+s}|\xi|^{\sigma}\big) \right) m\left(\frac{t-s}{1+s}, 0, \cdot\right).
\end{equation}
Now we may state the main result of this section.
\begin{proposition}\label{Prop1H}
Let $\mu\in [0,1)$,  $\sigma>1$, $1< p\leq 2\leq q <\infty$ and assume
\begin{equation}\label{regularity}
\frac{n}{\sigma}\Big( \frac1p-\frac1q \Big)- \frac{nd}{2} \leq 1-\frac{\mu}{2}.
\end{equation}
Then, the solution to \eqref{eq:Homogeneous}, with $w(\xi)=|\xi|^{\sigma}$, satisfies
\begin{equation*}
\|u(t, s, \cdot)\|_{L^q} \lesssim  (1+t)^{1-\mu-\frac{n}{\sigma}\left(\frac1p-\frac1q\right)} (1+s)^{\mu} \|u_1\|_{L^p\cap L^2}, \quad t\geq s.
\end{equation*}
\end{proposition}
To prove Proposition \ref{Prop1H}, we use the following dispersive estimates for oscillatory multipliers (see Theorems 2 and 5 in \cite{DEL=2026}).
\begin{lemma} \label{Lemma-DELH}
Let $1< p\leq 2\leq q<\infty$, $w(\xi)=|\xi|^{\sigma}, \sigma>1$ and $\eta\in \mathbb{R}$. Let $d=d(p,q)$ and $\kappa$ as in~\eqref{eq:dH} and \eqref{eq:kappaH}, respectively.
\begin{enumerate}[(i)]
\item We assume that $\kappa \geq0$. Then
\begin{equation} \label{eq:estlow-plate}
\| |\xi|^{\eta}e^{itw(\xi)}\chi(\xi) \|_{M_p^q} \leq (1+t)^{-\min\left\{ \frac{\kappa}{\sigma}, \frac{nd}2  \right\}}, \quad t\geq 0.
\end{equation}
By replacing $e^{it\omega}$ by $e^{it\omega}-e^{-it\omega}$, \eqref{eq:estlow-plate} remains valid even if we weaken our assumption~$\kappa\geq0$ to $\kappa\geq -\sigma$.
\item Moreover,
\begin{equation}\label{eq:esthi0deg-plate}
\| |\xi|^{\eta}e^{itw(\xi)}(1-\chi)(\xi) \|_{M_p^q} \leq \begin{cases}
C\,t^{-\frac{nd}2}, & t\geq1, \\
C\,t^{-\frac{\kappa}{\sigma}},& t\in(0,1],
\end{cases}
\end{equation}
provided that $\kappa/\sigma\leq nd/2$. Estimate \eqref{eq:esthi0deg-plate} also holds if $\kappa<0$ replacing $C\,t^{-\frac{\kappa}{\sigma}}$ by $C$.
\end{enumerate}
\end{lemma}
\begin{remark}
Proposition \ref{Prop1H} for $\mu=0$ is already well known (see for instance \cite{Ebert-L}).
\end{remark}
\begin{remark}
In the proof of Theorem \ref{thm:Nonlinear-plate-1}, we will apply the estimate obtained in Proposition \ref{Prop1H} with $p=1+\delta$, where $\delta>0$ is sufficiently small, and $2\leq q<\infty$. Hence, the restriction on $\mu$ in \eqref{regularity} becomes $0<\mu\leq 2-\frac{n}{\sigma}$. This interval is nonempty provided that $n<2\sigma$.
\end{remark}
\subsubsection{Considerations in $Z_{\text{high}}$}
Thanks to Proposition \ref{Prop-Hankel} and Lemma \ref{hankel}, we may write
\[ (1-\chi(w(\xi))m\left(\frac{t-s}{1+s}, 0, \cdot\right)=C(1-\chi(w(\xi)) a\left(\frac{1+t}{1+s}w(\xi)\right)\,b(w(\xi))\, e^{i\frac{1+t}{1+s}w(\xi)-iw(\xi)} ,\]
where $a, b\in S^{-\frac12}$.
Using  the homogeneity of  $w$ and Lemma \ref{lemma:m1-for-Plate-BQ}, for $p\in (1, \infty)$ we have
\begin{align*}
\|w(\xi)^{1/2} b(w(\xi))\|_{M_p^p} & \leq C ,
\end{align*}
and
\[ \big\| w(\xi)^{1/2} a\left(\frac{1+t}{1+s}w(\xi)\right) \big\|_{M_p^p}=\left(\frac{1+t}{1+s}\right)^{-1/2}\| w(\xi)^{1/2} a(w(\xi)) \|_{M_p^p} \leq C \left(\frac{1+t}{1+s}\right)^{-1/2}. \]
Applying Lemma \ref{Lemma-DELH} with  $\eta=-\sigma$,  for $1< p\leq 2\leq q<\infty$ we get
\[ \big\| |\xi|^{-\sigma}(1-\chi(w(\xi)) e^{i\frac{1+t}{1+s}w(\xi)-iw(\xi)}\big\|_{M_p^q} \leq
C\,\left(\frac{1+t}{1+s}-1\right)^{-\frac{nd}2},
\]
provided that $\frac{n}{\sigma}\left(\frac1p-\frac1q\right)- \frac{nd}{2}\leq 1$. Therefore, using again Proposition \ref{Prop-Hankel} and \eqref{simple},  we obtain
\begin{align*}
\big\|\mathcal{F}^{-1}\big( (1-\chi(w(\xi))\psi(t,s,\xi) \big)&\ast u_1\big\|_{L^q} \lesssim (1+t)^{-\frac{\mu}2-\frac{nd}2}
(1+s)^{1+\frac{\mu}2 +\frac{nd}2-\frac{n}{\sigma}\left(\frac1p-\frac1q\right)} \|u_1\|_{L^p\cap L^2}.
\end{align*}
In the last inequality, since $d=d(p,q)=0$ when $p=2$ for all $q\geq 2$, the possible singular behavior of the solution at $t=s$ can be avoided under the additional assumption $u_1\in L^2$.
\subsubsection{Considerations at low frequencies}
\begin{itemize}
\item \textbf{Step 1:} Let us consider the first term in \eqref{simple2}.

Due to Proposition \ref{Prop-Hankel} and Lemma \ref{hankel}, we have
\begin{align*}
&\left|\chi(w(\xi))\left( 1-\chi\big(\frac{1+t}{1+s}w(\xi)\big) \right)  m\left(\frac{t-s}{1+s}, 0, \cdot\right)\right|=C \chi(w(\xi))\left( 1-\chi\big(\frac{1+t}{1+s}w(\xi)\big) \right) \\
&\qquad \times \left| a^{+}\left(\frac{1+t}{1+s}w(\xi)\right)\, e^{i\frac{1+t}{1+s}w(\xi)} H_{\rho}^{-}(w(\xi))-a^{-}\left(\frac{1+t}{1+s}w(\xi)\right)\, e^{-i\frac{1+t}{1+s}w(\xi)}H_{\rho}^{+}(w(\xi)) \right|,
\end{align*}
where $a^{\pm}\in S^{-\frac12}$. We consider only the first component in $m$ (since the analysis for the second one is similar). We have
\begin{align*}
&\chi(w(\xi))\left( 1-\chi\big(\frac{1+t}{1+s}w(\xi)\big) \right)H_{\rho}^{-}(w(\xi)) \,a^{+}\left(\frac{1+t}{1+s}w(\xi)\right)\, e^{i\frac{1+t}{1+s}w(\xi)} =: m_1(s,\xi)m_2(t,\xi),
\end{align*}
where
\begin{align*}
m_1(\xi) &:= \chi(w(\xi))w(\xi)^{\rho}H^{-}_{\rho}(w(\xi)), \\
m_2(t,\xi) &:=\left( 1-\chi\left(\frac{1+t}{1+s}w(\xi)\right) \right)a^{+}\left(\frac{1+t}{1+s}w(\xi)\right)\,w(\xi)^{-\rho} e^{i\frac{1+t}{1+s}w(\xi)} \\
&= \underbrace{\left(\frac{1+t}{1+s}w(\xi)\right)^{\frac12}a^{+}\left(\frac{1+t}{1+s}w(\xi)\right)}_{=:m_3(t,\xi)}\underbrace{ \left(1- \chi\left(\frac{1+t}{1+s}w(\xi)\right) \right) w(\xi)^{-\rho}\big(\frac{1+t}{1+s}w(\xi)\big)^{-\frac{1}{2}}e^{i\frac{1+t}{1+s}w(\xi)}}_{=:m_4(t,s, \xi)}.
\end{align*}
Thanks to  Lemma \ref{lemma:m1-for-Plate-BQ}, $m_3(t, s,\cdot)\in M_p^p$. We will show (see Lemma \ref{lemmaM1}) that $m_1\in M_p^p$ with $p\in (1, \infty)$.  
 By homogeneity and by Lemma \ref{Lemma-DELH}, we have
\[\|m_4(t,s, \cdot)\|_{M_p^q} =\left(\frac{1+t}{1+s}\right)^{\frac{1}{2}-\frac{\mu}{2}-\frac{n}{\sigma}\left(\frac1p-\frac1q\right)}\|m_4(0,\cdot)\|_{M_p^q} \leq \left(\frac{1+t}{1+s}\right)^{\frac{1}{2}-\frac{\mu}{2}-\frac{n}{\sigma}\left(\frac1p-\frac1q\right)} , \]
for all $1< p\leq 2\leq q<\infty$ provided that $\frac{n}{\sigma}\left(\frac1p-\frac1q\right)- \frac{nd}{2}\leq 1-\frac{\mu}{2}$.
\begin{lemma}\label{lemmaM1}
Let $\mu\in(0,1)$ and define
\[ m_1(s,\xi) := \chi((1+s)w(\xi))((1+s)w(\xi))^{\rho} H^{-}_{\rho}((1+s)w(\xi)), \]
where $\chi\in \mathcal{C}_c^\infty([0,\infty))$ is a cut--off function supported in $[0,2]$. Then, for any $1<p<\infty$, $m_1(s,\cdot)\in M_p^p(\mathbb{R}^n)$ and $ \|m_1(s,\cdot)\|_{M_p^p}\leq C$ for all $s\geq 0$.
\end{lemma}
\begin{proof}
First we remark that by homogeneity $\|m_1(s,\cdot)\|_{M_p^p}=\|m_1(0,\cdot)\|_{M_p^p}$ and by using \eqref{Eq:Hankel-property}, we find
\[ F(r):=r^{\rho} H^{-}_{\rho}(r)= r^{\rho}(i\sin(\rho\pi))^{-1}\big(e^{i\rho \pi}  J_{\rho}(r)-J_{-\rho}(r) \big) \quad \text{with} \quad r:=w(\xi)
\]
is an entire function  and $|w^{(k)}(r)|\leq C r^{\sigma-k}$ for any $k\in\mathbb{N}_0$ and $r\in [0,2]$. Moreover, since $\chi$ is supported in $[0,2]$, by the chain rule, for any $\gamma$ a multi-index, there exists $C_\gamma>0$ such that
\[ |\partial_\xi^\gamma m_1(s,\xi)| \leq   C_\gamma |\xi|^{-|\gamma|}, \qquad \xi\neq 0.\]
Choosing $|\gamma|\le [n/2]+1$, the Mikhlin--H\"{o}rmander multiplier theorem yields
\[ m_1(s,\cdot)\in M_p^p(\mathbb{R}^n) \quad\text{for all }1<p<\infty, \]
with the operator norm uniformly bounded.
\end{proof}

\item \textbf{Step 2:} Let us consider the second term in \eqref{simple2}.

In this zone, since $H^{\pm}_{\rho}=J_{\rho} \pm i Y_{\rho}$ and $\rho\notin \mathbb{Z}$, we use the following representation for the multiplier:
\begin{equation*}\label{Bessel}
m\left(\frac{t-s}{1+s}, 0, \xi\right)=2i\csc(\rho \pi)
\left| \begin{array}{cc}
J_{-\rho}\left( w(\xi) \right) & J_{-\rho}\left( \frac{1+t}{1+s}w(\xi) \right) \\
J_{\rho}\left( w(\xi) \right) &  J_{\rho}\left( \frac{1+t}{1+s}w(\xi) \right) \\
\end{array} \right|,
\end{equation*}
where $J_{\rho}$ and $Y_{\rho}$ denote the Bessel functions of the first and second kind, respectively. Then, we obtain
\[ |\tilde m(t, s, \xi)| := \left|\chi(w(\xi))\chi\left(\frac{1+t}{1+s}w(\xi)\right)m\left(\frac{t-s}{1+s}, 0, \xi\right)\right| \lesssim  (1+s)^{-\rho}(1+t)^{\rho}+(1+s)^{\rho}(1+t)^{-\rho}. \]
Let $q'$ denotes the conjugate exponent of $q$, and $r$ verifies
\[ \frac{1}{r}=\frac{1}{q'}-\frac{1}{p'}=\frac{1}{p}-\frac{1}{q}, \quad 1\leq p\leq 2\leq q\leq\infty. \]
Then, using the Hausdorff-Young and H\"older's inequalities, we estimate the following:
\begin{align*}
\big\| \mathcal{F}^{-1}\big( \tilde m(t, s, \xi) \big)\ast u_1 \big\|_{L^q} & \lesssim \| \tilde m(t, s, \xi) \hat{u}_1\|_{L^{q'}} \lesssim \| \tilde m(t, s, \cdot)\|_{L^r} \|\hat u_1\|_{L^{p'}} \lesssim\,\left(\frac{1+t}{1+s}\right)^{-\frac{n}{\sigma}\left(\frac{1}{p}-\frac{1}{q} \right)+\rho}
\|u_1\|_{L^p}\, ,
\end{align*}
provided that $\mu\in (0,1)$. In the last inequality we used that
\[ \| \tilde m(t, s, \xi)\|^r_{L^r} = \int_{|\xi|\leq \left(\frac{1+t}{1+s}\right)^{-\frac{1}{\sigma}}}\, d\xi \lesssim  \left(\frac{1+t}{1+s}\right)^{-\frac{n}{\sigma}}. \]
\end{itemize}
Summing up the previous two steps and  using again Proposition \ref{Prop-Hankel} and \eqref{simple},  we obtain
\begin{align*}
\|\chi(w(\xi))\psi(t,s,\cdot)\|_{M_p^q} &=C(1+t)^{\rho}(1+s)^{1-\rho}\|\chi(w(\xi))m(t,s,\cdot)\|_{M_p^q} \\
& \lesssim (1+t)^{1-\mu-\frac{n}{\sigma}\left( \frac{1}{p}-\frac{1}{q} \right)} (1+s)^{\mu}
\end{align*}
for all $1<p\leq 2\leq q<\infty$ and $\frac{n}{\sigma}\left(\frac1p-\frac1q\right)- \frac{nd}{2}\leq 1-\frac{\mu}{2}$.

\begin{proof}[Proof of Proposition \ref{Prop1H}]
At low frequencies we concluded that
\[ \left\| \mathcal{F}^{-1}\big( \chi(w(\xi))\psi(t,s,\xi) \big)\ast u_1 \right\|_{L^q} \lesssim (1+s)^{\mu}(1+t)^{1 - \mu - \frac{n}{\sigma}(\frac{1}{p} - \frac{1}{q})}\, \|u_1\|_{L^p}. \]
Hence is sufficient to remark that if $\mu\leq 2-\frac{2n}{\sigma}\left(\frac1p-\frac1q\right)+nd$,  the derived  estimate at high frequencies are better than the one at low frequencies, namely
\[(1+t)^{-\frac{\mu}2-\frac{nd}2}
(1+s)^{1+\frac{\mu}2 +\frac{nd}2-\frac{n}{\sigma}\left(\frac1p-\frac1q\right)}\leq (1+t)^{1 - \mu - \frac{n}{\sigma}(\frac{1}{p} - \frac{1}{q})}(1+s)^{\mu}.
\]
\end{proof}
\subsection{Linearized Boussinesq equation with time dependent damping}\label{sec:linear}
In this section, we consider the Cauchy problem for the linearized Boussinesq equation with time dependent damping
\begin{equation}\label{eq:LBQ}
\begin{cases}
u_{tt}-\Delta u +\Delta^2 u+ \dfrac{\mu}{1+t} u_t =0, & t\geq s\geq 0\,\,\, x\in\R^n,\\
u(s,x)=0, \quad u_t(s,x)=u_1(s, x), & x\in\mathbb{R}^n.
\end{cases}
\end{equation}
\begin{proposition}\label{Prop1}
Let $\mu\in (0,1)$, $1< p\leq 2\leq q <\infty$, $\nu,\nu_0\in \mathbb{R}$ and assume
\begin{equation*}
n\Big(\frac1p-\frac1q\Big)+\nu-\nu_0 \leq nd.
\end{equation*}
Then, the solution to \eqref{eq:LBQ} satisfies
\begin{equation*}
\|u(t, s, \cdot)\|_{L^q} \leq (1+t)^{-\frac{\mu}2}(1+s)^{\frac{\mu}2}\|\,|D|^{-1-\nu}\<D\>^{\nu_0-1}u_1\|_{L^p}
\begin{cases}
C\,(t-s)^{-\frac{n}{2}\left(\frac1p-\frac1q\right)-\frac{\nu-\nu_0}{2}}, & t-s\in (0, 1], \\
C\,(1+t-s)^{-\delta_\BQ}, & t-s>1,
\end{cases}
\end{equation*}
where we set
\[ \delta_\BQ=\min\Big\{ n\big(\frac1p-\frac1q\big)+\nu\,,\, \frac13\Big(n\big(\frac1p-\frac1q\big)+\nu+(n-1)d\Big)\,,\, \frac{nd}{2}\Big\}. \]
\end{proposition}
\begin{remark}\label{remarkKey} Let $\kappa_{\low}=n\left(\frac1p-\frac1q\right)+\nu$. It is easy to verify that
\begin{equation*} \label{eq:deltaBQ}
\begin{split}
\delta_\BQ
    & = \min\left\{\kappa_\low, \frac{\kappa_\low}3+\frac{(n-1)}{3}d, \frac{nd}2\right\} = \begin{cases}
\kappa_\low, & \kappa_\low\leq\frac{n-1}2 \,d,\\
\frac{\kappa_\low}{3}+\frac{(n-1)}{3}d, & 0\leq \kappa_\low\leq \left(\frac{n}2+1\right) d,\\
\frac{nd}{2}, & \kappa_\low\geq \left(\frac{n}2 +1\right) d.
\end{cases}
\end{split}
\end{equation*}
We are interested in the case whether $\delta_\BQ=\frac{nd}2$, hence, for $q=p'$ we assume the condition
\[\kappa_\low=nd+\nu\geq \left(\frac{n}2+1\right) d.\]
This condition holds for all  $\nu\geq 0$ if $n\geq 2$ and for $\nu\geq 1/2$ if $n= 1$.
\end{remark}
%
From Theorems 2 and 5 in \cite{DEL=2026}, we have the following lemma.
\begin{lemma} \label{Lemma-DEL}
Let $1\leq p\leq 2\leq q\leq\infty$, $w(\xi)=|\xi|\sqrt{1+|\xi|^2}$ and $\eta\in \mathbb{R}$.  Let $d=d(p,q)$ and $\kappa$ as in~\eqref{eq:dH} and \eqref{eq:kappaH}, respectively. We assume that $\kappa >0$. Then
\begin{equation} \label{eq:estlow}
\| |\xi|^{\eta}e^{itw(\xi)}\chi(\xi) \|_{M_p^q} \leq (1+t)^{-\delta_\BQ},
\end{equation}
where we set
\[ \delta_\BQ=\min\Big\{ \kappa\,,\, \frac13\big(\kappa+(n-1)d\big)\,,\, \frac{nd}{2} \Big\}. \]
Estimate \eqref{eq:estlow} also holds if $\kappa=0$ and $1<p\leq 2\leq q<\infty$.\\
Moreover,
\begin{equation*}\label{eq:esthi0deg}
\| |\xi|^{\eta}e^{itw(\xi)}(1-\chi)(\xi) \|_{M_p^q} \leq \begin{cases}
C\,t^{-\frac{nd}2}, & t\geq1, \\
C\,t^{-\frac{\kappa}{2}},& t\in(0,1],
\end{cases}
\end{equation*}
provided that $\kappa\leq nd$.
\end{lemma}
In the following, for both low and high frequencies, we estimate the function $g(\xi)=w(\xi)^{-1}$ by a suitable power $|\xi|^{\eta}$ in $M_p^p$, in the sense that $\|g\,|\xi|^{-\eta}\chi\|_{M_p^p}\leq C$. This is straightforward when $1<p<\infty$, since the result follows directly from the Mikhlin--H\"ormander theorem. An analogous estimate also holds when $\chi$ is replaced by $1-\chi$.
\subsubsection{Considerations in $Z_{\text{high}}$}
In $Z_{\text{high}}$, we have $w(\xi)\sim |\xi|^2$. Thus, from Proposition \ref{Prop-Hankel} and Lemma \ref{hankel}, we get
\begin{align*}
\big(1-\chi(w(\xi)\big)\psi(t,s,\xi) &= C\big(1-\chi(w(\xi)\big)(1+t)^{\rho}(1+s)^{1-\rho}a((1+t)w(\xi))\,b((1+s)w(\xi))\, e^{i(t-s)w(\xi)} \\
&= C(1+t)^{\rho}(1+s)^{1-\rho}\underbrace{((1+t)w(\xi))^{\frac12}a((1+t)w(\xi))}_{\in M_p^p}\underbrace{((1+s)w(\xi))^{\frac12}b((1+s)w(\xi))}_{\in M_p^p} \\
& \qquad \times \big(1-\chi(w(\xi)\big)|\xi|^{\nu-\nu_0}e^{i(t-s)w(\xi)}|\xi|^{-2-\nu+\nu_0}(1+t)^{-\frac12}(1+s)^{-\frac12},
\end{align*}
where $a, b\in S^{-\frac12}$ and we used Lemma \ref{lemma:m1-for-Plate-BQ}. Thus, applying Lemma \ref{Lemma-DEL} with $\eta=\nu-\nu_0$, we find
\[ \big\| |\xi|^{\nu-\nu_0}e^{itw(\xi)}\big(1-\chi(w(\xi)\big) \big\|_{M_p^q} \lesssim
\begin{cases}
(1+t-s)^{-\frac{nd}2}, & t-s\geq1, \\
(t-s)^{-\frac{\kappa}{2}},& t-s\in(0,1].
\end{cases}
\]
Thus, we obtain
\begin{align*}
\big\|\mathcal{F}^{-1}\big( 1-\chi&(w(\xi))\psi(t,s,\xi) \big)\ast u_1\big\|_{L^q} \\
& \lesssim (1+t)^{-\frac{\mu}2}(1+s)^{\frac{\mu}2} \|\<D\>^{\nu_0-\nu-2}u_1\|_{L^p}\begin{cases}
(1+t-s)^{-\frac{nd}2}, & t-s\geq1, \\
(t-s)^{-\frac{\kappa}{2}},& t-s\in(0,1].
\end{cases}
\end{align*}
provided that $n\big(\frac1p-\frac1q\big)+\nu-\nu_0 \leq nd$.
\subsubsection{Considerations in $Z_{1}$}
In $Z_1$, we  have $w(\xi)\sim |\xi|$. Therefore, from Proposition \ref{Prop-Hankel} and Lemma \ref{hankel}, we have
\begin{align*}
\psi(t,s,\xi) &= C(1+t)^{\rho}(1+s)^{1-\rho}a((1+t)w(\xi))\,b((1+s)w(\xi))\, e^{i(t-s)w(\xi)} \\
&= C(1+t)^{\rho}(1+s)^{1-\rho}\underbrace{((1+t)w(\xi))^{\frac12}a((1+t)w(\xi))}_{\in M_p^p}\underbrace{((1+s)w(\xi))^{\frac12}b((1+s)w(\xi))}_{\in M_p^p} \\
& \qquad \times |\xi|^{\nu}e^{i(t-s)w(\xi)}|\xi|^{-\nu-1}(1+t)^{-\frac12}(1+s)^{-\frac12},
\end{align*}
where $a, b\in S^{-\frac12}$ and we used Lemma \ref{lemma:m1-for-Plate-BQ}. Multiplying by the localization function and using multiplier properties, we obtain
\[ \big\|\chi(w(\xi))\big(1-\chi((1+s)w(\xi))\big)\psi(t,s,\xi)\big\|_{M_p^q}
\leq \|\chi(w(\xi))\psi(t,s,\xi)\|_{M_p^q}, \]
thanks to
\[ \|1-\chi((1+s)w(\xi))\|_{M_p^p}\leq C \]
for all $p\in(1,\infty)$, uniformly in $s\ge0$. Indeed, let
\[ h(s,\xi):=1-\chi((1+s)w(\xi)), \qquad s\ge0. \]
By construction of $\chi$, one has
\[ h(s,\xi)=0 \quad \text{if}\,\,\, (1+s)w(\xi)\leq1 \qquad \text{and} \qquad h(s,\xi)=1 \quad \text{if}\,\,\, (1+s)w(\xi)\geq2. \]
Hence, for every multi-index $\alpha\neq0$, the support of $\partial_\xi^\alpha h(s,\xi)$ is contained in the set
\[ \Omega_s:=\{\xi\in\mathbb R^n:\ 1<(1+s)w(\xi)<2\}. \]
By Fa\`a di Bruno's formula, for every multi-index $\alpha\neq0$,
\[ \partial_\xi^\alpha h(s,\xi) =
\sum_{k=1}^{|\alpha|}
\chi^{(k)}((1+s)w(\xi))(1+s)^k
\sum_{\substack{\beta_1+\cdots+\beta_k=\alpha\\ |\beta_j|\ge1}} C_{\alpha,\beta} \prod_{j=1}^{k}\partial_\xi^{\beta_j}w(\xi),
\]
where $C_{\alpha,\beta}$ are suitable combinatorial constants. Since $\chi^{(k)}((1+s)w(\xi))=0$ unless $\xi\in\Omega_s$, and
\[ |\partial_\xi^\gamma w(\xi)|\le C_\gamma |\xi|^{1-|\gamma|}, \qquad \gamma\in\mathbb N^n, \]
we infer that each summand is bounded by
\[ C(1+s)^k \prod_{j=1}^{k}|\xi|^{1-|\beta_j|} = C(1+s)^k |\xi|^{k-|\alpha|} \lesssim |\xi|^{-k}|\xi|^{k-|\alpha|} = |\xi|^{-|\alpha|}, \]
where we used $(1+s)\sim |\xi|^{-1}$ on $\Omega_s$. Therefore,
\[ |\partial_\xi^\alpha h(s,\xi)| \leq C_\alpha |\xi|^{-|\alpha|}, \quad \alpha\in\mathbb N^n, \]
with constants independent of $s\ge0$. This shows that $h=h(s,\xi)$ satisfies the Mikhlin--H\"ormander condition uniformly with respect to $s$, and hence $\|h(s,\cdot)\|_{M_p^p}\leq C$, $p\in(1,\infty)$, with $C$ independent of $s$.

Thus, applying Lemma \ref{Lemma-DEL} with $\eta=\nu$ and $\kappa>0$, we obtain
\[ \big\| |\xi|^{\nu}e^{i(t-s)w(\xi)}\chi(w(\xi)\|_{M_p^q} \lesssim (1+t-s)^{-\delta_{\mathrm{BQ}}}, \]
and
\begin{align*}
\big\|\mathcal{F}^{-1}\big(\chi(w(\xi))(1-\chi&((1+s)w(\xi)))\psi(t,s,\xi)\big)\ast u_1\big\|_{L^q} \\
& \lesssim\, (1+t)^{-\frac{\mu}{2}}(1+s)^{\frac{\mu}{2}}
(1+t-s)^{-\delta_{\mathrm{BQ}}} \|\,|D|^{-\nu-1}u_1\|_{L^p}, \quad t\geq s.
\end{align*}
\subsubsection{Considerations in $Z_{2}$}
By Proposition \ref{Prop-Hankel} and Lemma \ref{hankel}, we have
\[ |\chi(w(\xi))\chi_2(t, s, \xi)\psi(t,s,\xi)| =C(1+t)^{\rho}(1+s)^{1-\rho}\chi(w(\xi))\chi_2(t, s, \xi) m(t,s,\xi) , \]
and
\[ m(t,s,\xi)=\big| a^{+}((1+t)w(\xi))\, e^{i(1+t)w(\xi)} H_{\rho}^{-}((1+s)w(\xi))-a^{-}((1+t)w(\xi))\, e^{-i(1+t)w(\xi)} H_{\rho}^{+}((1+s)w(\xi)) \big|, \]
where $a^{\pm}\in S^{-\frac12}$. We consider only the first component in $m$ (since the analysis for the second one is similar). We have
\begin{align*}
\chi(w(\xi))\chi_2&(t, s, \xi)\, |\xi|^{1+\nu}\, m(t,s,\xi) \\
&= \chi_2(t, s, \xi) H^{-}_{\rho}((1+s)w(\xi))|\xi|^{1+\nu}\,a((1+t)w(\xi))\, e^{i(1+t)w(\xi)}  \\
&= \chi_2(t, s, \xi)((1+s)w(\xi))^{\rho-\rho}H^{-}_{\rho}((1+s)w(\xi))|\xi|^{1+\nu}\,a((1+t)w(\xi))\, e^{i(1+t)w(\xi)}  \\
&= m_1(s,\xi)m_2(t,\xi)(1+s)^{-\rho-\frac{\mu}{2}},
\end{align*}
where
\begin{align*}
m_1(s,\xi) &:= \chi((1+s)w(\xi))((1+s)w(\xi))^{\rho+\frac{\mu}{2}}H^{-}_{\rho}((1+s)w(\xi)) \\
m_2(t,\xi) &:= \big(1- \chi((1+t)w(\xi)) \big) \chi(w(\xi))|\xi|^{1+\nu} w(\xi)^{-\frac{1}{2}}a((1+t)w(\xi))\,e^{i(1+t)w(\xi)} \\
&= \underbrace{\big( 1-\chi((1+t)w(\xi)) \big)((1+t)w(\xi))^{\frac12} a((1+t)w(\xi))}_{=:m_3(t,\xi)} \, \chi(w(\xi))|\xi|^{1+\nu} w(\xi)^{-\frac{1}{2}}((1+t)w(\xi))^{-\frac{1}{2}}e^{i(1+t)w(\xi)} \\
&= m_3(t,\xi) \chi(w(\xi))|\xi|^{1+\nu} w(\xi)^{-1}(1+t)^{-\frac{1}{2}}e^{i(1+t)w(\xi)}.
\end{align*}
Let
\[ m_4(t,\xi) := \chi(w(\xi))|\xi|^{1+\nu} w(\xi)^{-1}(1+t)^{-\frac{1}{2}}e^{i(1+t)w(\xi)}. \]
Since $\chi(w(\xi))w(\xi)\sim |\xi|$, by Lemma \ref{Lemma-DEL}, with $\eta=\nu$, we have
\[ \|m_4(t,\cdot)\|_{M_p^q} \leq (1+t)^{-\frac{1}{2}-\delta_{\text{BQ}}}. \]
Thanks to  Lemma 3.2,   $ \|m_3(t,\cdot)\|_{M_p^p}\leq C$ for all $t\geq 0$, whereas following as in the proof of Lemma \ref{lemmaM1} we conclude that
$\|m_1(s,\cdot)\|_{M_p^p}\leq C$ for all $s\geq 0$. Therefore, it holds
\[ (1+t)^{\rho}(1+s)^{1-\rho} \| \chi(w(\xi))\chi_2(t, s, \xi)m(t,s,\cdot)\|_{M_p^q} \lesssim (1+t)^{-\frac{\mu}{2}-\delta_{\text{BQ}}}(1+s)^{\frac{\mu}{2}}. \]
\subsubsection{Considerations in $Z_{3}$}
In this zone, since $H^{\pm}_{\rho}=J_{\rho} \pm i Y_{\rho}$ and $\rho\notin \mathbb{Z}$, we use the following representation for the multiplier:
\begin{equation*}\label{Bessel-BQ}
m(t, s, \xi)=2i\csc(\rho \pi)
\left| \begin{array}{cc}
J_{-\rho}\left( (1+s)w(\xi) \right) & J_{-\rho}\left( (1+t)w(\xi) \right) \\
J_{\rho}\left( (1+s)w(\xi) \right) &  J_{\rho}\left( (1+t)w(\xi) \right) \\
\end{array} \right|,
\end{equation*}
where $J_{\rho}$ and $Y_{\rho}$ denote the Bessel functions of the first and second kind, respectively. Then, we obtain
\[ |\chi_3(t, s,\xi) m(t, s, \xi)| \lesssim (1+s)^{-\rho}(1+t)^{\rho} + (1+s)^{\rho}(1+t)^{-\rho}, \]
so that
\[ |\chi_3(t, s,\xi) \psi(t,s,\xi)|\lesssim (1+s)^{1-2\rho}(1+t)^{2\rho} + (1+s). \]
When $\rho\not\in \mathbb{Z}$, using the Hausdorff--Young and Hölder inequalities, and setting
\[ \frac{1}{r}=\frac{1}{p}-\frac{1}{q}, \qquad 1\le p\le 2\le q\le\infty, \]
we estimate
\begin{align*}
\big\|\mathcal{F}^{-1}\big(\chi_3(t,s,\xi)\psi(t,s,\xi)\big)\ast u_1\big\|_{L^q}
&= \big\|\mathcal{F}^{-1}\big(\chi_3(t,s,\xi)|\xi|^{1+\nu}\psi(t,s,\xi)\big)\ast |D|^{-1-\nu}u_1\big\|_{L^q} \\
&\lesssim \|\chi_3(t,s,\xi)|\xi|^{1+\nu}\psi(t,s,\xi)\|_{L^r}\, \||D|^{-1-\nu}u_1\|_{L^p}.
\end{align*}
By the support property of $\chi_3$, we have $|\xi|\le (1+t)^{-1}$, and hence
\[ \|\chi_3(t,s,\xi)|\xi|^{1+\nu}\|_{L^r}^r = \int_{|\xi|\le (1+t)^{-1}}|\xi|^{(1+\nu)r}\,d\xi \lesssim (1+t)^{-n-(1+\nu)r}, \]
which implies
\[ \|\chi_3(t,s,\xi)|\xi|^{2}\|_{L^r} \lesssim (1+t)^{-n\left(\frac{1}{p}-\frac{1}{q}\right)-1-\nu}. \]
Consequently, using the bound on $\psi(t,s,\xi)$, we arrive at
\[ \big\|\mathcal{F}^{-1}\big(\chi_3(t,s,\xi)\psi(t,s,\xi)\big)\ast u_1\big\|_{L^q} \lesssim (1+t)^{-n\left(\frac{1}{p}-\frac{1}{q}\right)-1-\nu+\rho+|\rho|}(1+s)^{1-\rho-|\rho|} \,\||D|^{-1-\nu}u_1\|_{L^p}. \]
Since $\mu<1$, we conclude the following $L^p$--$L^q$ estimates:
\begin{align*}
\big\|\mathcal{F}^{-1}\big( \chi_3(t,s,\xi)\psi(t,s,\xi) \big)\ast u_1\big\|_{L^q}
\lesssim (1+s)^{\mu} (1+t)^{-\mu-n\left(\frac{1}{p}-\frac{1}{q}\right)-\nu} \,\||D|^{-1-\nu}u_1\|_{L^p}.
\end{align*}
\begin{proof}[Proof of Proposition \ref{Prop1}]
It is sufficient to remark that the derived estimates in $Z_2, Z_3$ are better than the one in $Z_1$. Indeed, we may estimate
\begin{align*}
(1+t)^{-\mu-\delta_{\text{BQ}}}(1+s)^{\mu} &\leq (1+t)^{-\frac{\mu}2-\delta_{\text{BQ}}}(1+s)^{\frac{\mu}2}.
\end{align*}
\end{proof}
\section{Proof of global in time existence results} \label{Sec:Proofs}
Let us begin by introducing some notations that will be used throughout this section. First, we denote by $u^{\lin}=u^{\lin}(t,x)$ the solution to the linear Cauchy problem \eqref{Eq:LinearProblem} with $s=0$, and by $K=K(t,\tau,x)$ the corresponding fundamental solution. Next, we define the operator $N$ as follows:
\[ N: u\in X(T) \rightarrow Nu(t,x) =  u^{\lin}(t,x)+u^{\nl}(t,x), \]
where $X(T)$ is a suitable function space in which we seek solutions, and
\begin{align*}
u^{\nl}(t,x) &= \int_0^tK(t,\tau,x)\ast_{(x)}f(u)d\tau,
\end{align*}
where $f(u)=|u(\tau,x)|^\alpha$ or $f(u)=\Delta|u(\tau,x)|^\alpha$. By Duhamel's principle, the solution to \eqref{Eq:SemilinearProblem} can be represented as a fixed point of the operator $N$.

To establish the existence of such a fixed point, we will prove that $N$ satisfies the following estimates:
\begin{align}
\label{Eq:inequality-Nu} \|Nu\|_{X(T)} &\leq C_0\| u_1\|_{\mathcal{A}} + C_1(t)\|u\|_{X(T)}^\alpha, \\
\label{Eq:inequality-NuNv} \|Nu-Nv\|_{X(T)} &\leq C_2(t)\|u-v\|_{X(T)}\big( \|u\|_{X(T)}^{\alpha-1}  + \|v\|_{X(T)}^{\alpha-1} \big),
\end{align}
where $\mathcal{A}$ denotes the function space of the initial data $u_1$. These estimates enable us to apply the contraction mapping principle, thereby obtaining a unique solution $u\in X(T)$ to the equation $Nu=u$. In particular,
\begin{itemize}
\item for large initial data, we obtain local (in time) existence by exploiting the fact that $C_1(t), C_2(t)\to 0$ as $t\to 0$;
\item for small initial data, we obtain global-in-time existence provided that $\max\{C_1(t)\,,\, C_2(t)\}\leq C$ for all $t\in [0, \infty)$, where $C$ is a suitable constant.
\end{itemize}
In the  next results the following lemma comes into play.
\begin{lemma}\label{Lem:Auxiliary-GlobalExistence}
Let $\nu >-1$ and $\beta \in \mathbb{R}$. Then, it holds
\begin{equation*}
\int_{0}^{t}(t-s)^{\nu}\,(1+s)^{\beta}\,ds\lesssim
\begin{cases}
(1+t)^{\nu}, &   \mbox{ if }\,\, \beta <-1,\\
(1+t)^{\nu}\,log(e+t), &  \mbox{ if }\,\, \beta =-1,\\
(1+t)^{1+\nu+\beta}, &  \mbox{ if }\,\,  \beta >-1,
\end{cases}
\end{equation*}
and
\[\int_{0}^{t}(t-s)^{\nu}\,e^{-c(t-s)}(1+s)^{\beta}\,ds\lesssim (1+t)^\beta .\]
Moreover, the estimate is also valid if $(t-s)^{\nu}$ is replaced by $(1+t-s)^{\nu}$ in the integral.
\end{lemma}
\subsection{Application to the nonlinear plate-type equation}
In this section, we consider the Cauchy problem for semilinear plate-type equation
\begin{equation*} \label{eq:NP}
\begin{cases}
u_{tt} + (-\Delta)^\sigma u + \dfrac{\mu}{1+t}u_t = |u|^\alpha & t>0,\ x\in\mathbb{R}^n,\\
u(0,x)=0, \quad u_t(0,x)=u_1(x), & x\in\mathbb{R}^n,
\end{cases}
\end{equation*}
where $\sigma\geq 2$ and $\alpha>1$.
\begin{proof}[Proof of Theorem \ref{thm:Nonlinear-plate-1}]
We define the solution space
\[  X(T) = \big\{u\in \mathcal{C}\big( [0,\infty), L^2(\mathbb{R}^n)\cap L^{r}(\mathbb{R}^n) \big): \ \|u\|_{X(T)}<\infty \big\}, \quad \text{where} \quad  r\in (\alpha,\infty). \]
The corresponding norm
\[ \|u\|_{X(T)} = \sup_{t\in[0,T],\, q\in[2,r]}(1+t)^{\mu-1+\frac{n}{\sigma}\left( \frac{1}{1+\delta}-\frac{1}{q} \right)}\|u\|_{L^q}. \]
Since $u_1 \in L^1(\R^n)\cap L^2(\R^n) = \mathcal{A}$, from Proposition \ref{Prop1H}, for $0<\mu <\min\left\{1, 2-\frac{n}{\sigma}\right\}$ we immediately conclude
\[ \|u^{\lin}\|_{X(T)} = \|K(t,0,x)\ast u_1(x)\|_{X(T)} \leq C_0\|u_1\|_{\mathcal{A}}. \]
Therefore, to conclude \eqref{Eq:inequality-Nu}, it remains to prove
\[ \|u^{\nl}\|_{X(T)} \lesssim \|u\|_{X(T)}^\alpha.\]
From Proposition \ref{Prop1H}, for  $0<\mu < \min\left\{1, 2-\frac{n}{\sigma}\right\}$ we have for all $2\leq q\leq r$
\begin{align*}
\|u^{\nl}(t,\cdot)\|_{L^q} &\lesssim \int_0^t \|K(t,s,\cdot)\ast |u(s,\cdot)|^\alpha\|_{L^q}\,ds \\
&\lesssim \int_0^{t} (1+t)^{1-\mu-\frac{n}{\sigma}\left(\frac{1}{1+\delta}-\frac{1}{q}\right)} (1+s)^\mu  \|\,|u(s,\cdot)|^\alpha\|_{L^{1+\delta}}\,ds \\
&\lesssim (1+t)^{1-\mu-\frac{n}{\sigma}\left(\frac{1}{1+\delta}-\frac{1}{q}\right)} \int_0^{t} (1+s)^\mu \|u(s,\cdot)\|_{L^{\alpha(1+\delta)}}^\alpha \,ds \\
& \lesssim (1+t)^{1-\mu-\frac{n}{\sigma}\left(\frac{1}{1+\delta}-\frac{1}{q}\right)} \|u\|_{X(T)}^\alpha \int_0^{t} (1+s)^{\mu+\alpha\left(1-\mu-\frac{n}{\sigma}\left(\frac{1}{1+\delta}-\frac{1}{\alpha(1+\delta)}\right)\right)} ds.
\end{align*}

The last integral is finite under the condition
\[ \mu+\alpha\Big(1-\mu-\frac{n}{\sigma(1+\delta)}\left(1-\frac{1}{\alpha}\Big) \right)<-1, \qquad \text{that is,} \qquad \alpha>1+\frac{2}{\mu+\frac{n}{\sigma}-1}. \]
Therefore, we conclude
\[ \|u^{\nl}(t,\cdot)\|_{L^q}
\lesssim (1+t)^{1-\mu-\frac{n}{\sigma}\left(\frac{1}{1+\delta}-\frac{1}{q}\right)} \|u\|_{X(T)}^\alpha. \]
Next we turn to the proof of the Lipschitz condition \eqref{Eq:inequality-NuNv}. Due to the fact that
\[ |\,|u(\tau,x)|^\alpha - |v(\tau,x)|^\alpha| \lesssim |u(s,x)-v(s,x)|\big( |u(s,x)|^{\alpha-1}+|v(s,x)|^{\alpha-1} \big), \]
from H\"{o}lder's inequality we obtain
\begin{align*}
\big\| |u(s,\cdot)|^\alpha-|v(s,\cdot)|^\alpha \big\|_{L^{1+\delta}} &\lesssim \| u(s,\cdot)-v(s,\cdot) \|_{L^{\alpha(1+\delta)}} \big( \|u(s,\cdot)\|_{L^{\alpha(1+\delta)}}^{\alpha-1}+\|v(s,\cdot)\|_{L^{\alpha(1+\delta)}}^{\alpha-1} \big).
\end{align*}
Thus, following the same ideas to what we did to estimate  $\|u^{\nl}(t,\cdot)\|_{L^q}$, we can obtain \eqref{Eq:inequality-NuNv}. In this way the proof of Theorem \ref{thm:Nonlinear-plate-1} is completed.
\end{proof}
\subsection{Application to the Boussinesq equation}\label{sec:nonlinear}
In this section, we consider the Cauchy problem for semilinear Boussinesq equation
\begin{equation} \label{eq:NBQ}
\begin{cases}
u_{tt}-\Delta u +\Delta^2 u+ \dfrac{\mu}{1+t} u_t =\Delta|u|^\alpha & t>0,\ x\in\R^n,\\
u(0,x)=0, \quad u_t(0,x)=u_1(x), & x\in\R^n,
\end{cases}
\end{equation}
where $\alpha>1$. We define the operator
\[ (Fu)(t,\cdot) =\int_0^t \Delta\, E(t-s)\ast f(u(s,\cdot))\,ds, \]
where $E=E(t)$ is the fundamental solution to the linear problem \eqref{eq:LBQ}, i.e., $v(t,s, \cdot)= E(t,s, \cdot)\ast u_1$.

If $n\geq 2$, applying~ Proposition \ref{Prop1}, with $s=0=\nu$ and $\nu_0=nd$, we get
\begin{equation*} \label{eq:crucial}
\| v(t, \cdot) \|_{L^q} \leq C\,(1+t)^{-\delta_{BQ}-\frac{\mu}{2}} \|\<D\>^{nd-1}|D|^{-1}u_1\|_{L^p}.
\end{equation*}
  If $n=1$, we apply
Proposition \ref{Prop1}, with $s=0$, $\nu=1/2$ and $\nu_0=3/2$, to get
\begin{equation*}
\| v(t, \cdot) \|_{L^q} \leq C\,(1+t)^{-\delta_{BQ}-\frac{\mu}{2}} \|\<D\>^{\frac12}|D|^{-\frac32}u_1\|_{L^p}.
\end{equation*}
Thanks to Remark \ref{remarkKey},  $\delta_\BQ=\frac{nd}2$.
Therefore, to maximize the decay rate we shall fix $p$ and $q$ on the conjugate line, that is,
\[ p=1+\frac1\alpha, \quad q=1+\alpha, \qquad \text{hence}, \qquad d(p,q)=\frac{\alpha-1}{\alpha+1}\,. \]
The critical exponent $\alpha_{\text{crit}}$ for the existence of global-in-time solutions is obtained when $\alpha_{\text{crit}}$ solves
\begin{equation*} \label{eq:tauStrauss}
\alpha\Big( \frac{nd}2+\frac{\mu}2 \Big)=1+\frac{\mu}2, \qquad \text{i.e.}, \qquad \frac{n}2\frac{\alpha-1}{\alpha+1}+\frac{\mu}{2}=\frac1\alpha\Big( 1+\frac{\mu}{2} \Big).
\end{equation*}
\begin{proof}[Proof of Theorem \ref{thm:NBQ}]
We define the solution space
\[  X = \big\{u\in L^\infty([0,\infty), L^{\alpha+1}): \ \|u\|_X\leq R \big\} \]
for $R>0$ that will be fixed later, and the corresponding norm
\[ \|u\|_X = \sup_{t>0} (1+t)^{\delta_\BQ+\frac{\mu}2}\|u(t)\|_{L^{\alpha+1}}, \quad \delta_\BQ=\frac{n}2\frac{\alpha-1}{\alpha+1}. \]
Let $E(t,s, \cdot)\ast u_1$ be the solution to the linear problem~\eqref{eq:LBQ}.
We notice that $nd<2$ is equivalent to $\alpha<1 +4/(n-2)$ when $n\geq3$, so that from the estimate above we derive
\begin{equation*} \label{eq:decayBQ-1}
\|E(t,0, \cdot)\ast u_1\|_{L^q} \leq C_1 (1+t)^{-\delta_\BQ-\frac{\mu}2} \,\|u_1\|_{\mathcal A}.
\end{equation*}
In particular, we fix $R=2C_1\|u_1\|_{\mathcal A}$, so that $\|E(t,0, \cdot)\ast u_1\|_X \leq R/2$. We now prove that the operator
\[ (Nu)(t,\cdot)= E(t,0, \cdot)\ast u_1(t,\cdot) + (Fu)(t,\cdot) \]
is a contraction on~$X$. Let $g(\tau,\cdot)=f(u(\tau,\cdot))-f(v(\tau,\cdot))$. We modify the estimate to reduce the regularity of $\Delta g$, namely, in Proposition \ref{Prop1} we set  $\nu=\nu_0=1$ at low frequencies and   $\nu=\nu_0=0$ at high frequencies, to get
\[ \| E(t, s, \cdot)\ast \Delta g(\tau,\cdot)\|_{L^{\alpha+1}} \leq (1+t)^{-\frac{\mu}2}(1+s)^{\frac{\mu}2} \begin{cases}
C\,(1+t-s)^{-\delta_\BQ}\,\||D|^{-2} \Delta g(\tau,\cdot)\|_{L^{1+\frac1\alpha}}\,, & t-s\geq1,\\
C\,(t-s)^{-\frac{nd}2}\,\|\<D\>^{-2} \Delta g(\tau,\cdot)\|_{L^{1+\frac1\alpha}}\,, & t-s\in(0,1].
\end{cases} \]
Recalling that $nd<2$, the singularity $(t-s)^{-\frac{nd}2}$ is integrable. Hence,
\[ \begin{split}
& \int_0^{t-1} \| E(t,s, \cdot)\ast \Delta g(s, \cdot)\|_{L^{\alpha+1}} ds \leq C\,(1+t)^{-\frac{\mu}2}\int_0^{t-1} (1+s)^{\frac{\mu}2} (t-s)^{-\delta_\BQ}\,\|g(s,\cdot)\|_{L^{1+\frac1\alpha}}\,ds,\\
& \int_{t-1}^t \| E(t, s, \cdot))\ast \Delta g(s,\cdot)\|_{L^{\alpha+1}} ds \leq C\,(1+t)^{-\frac{\mu}2}\int_{t-1}^t (1+s)^{\frac{\mu}2} (t-s)^{-\frac{nd}2}\,\|g(s,\cdot)\|_{L^{1+\frac1\alpha}}\,ds.
\end{split} \]
Since $u,v\in X$, we get
\[ \begin{split}
\|\<D\>^{-2}\Delta g(s,\cdot)\|_{L^{1+\frac1\alpha}}
    & \leq \||D|^{-2}\Delta g(s,\cdot)\|_{L^{1+\frac1\alpha}} \approx \|g(s,\cdot)\|_{L^{1+\frac1\alpha}} \\
    & \leq C\,R^{\alpha-1} \|u-v\|_X\,(1+s)^{-\alpha\left(\delta_\BQ+\frac{\mu}2\right)}.
\end{split} \]
We notice that
\[ \delta_\BQ \leq \frac{nd}2<1, \]
hence, using $\alpha\big(\delta_\BQ+\frac{\mu}2\big)>1+\frac{\mu}2$, from Lemma \ref{Lem:Auxiliary-GlobalExistence} we immediately obtain
\[ \begin{split}
& \int_0^{t-1} (t-s)^{-\delta_\BQ}\,(1+s)^{-\alpha\left(\delta_\BQ+\frac{\mu}2\right)+\frac{\mu}2}\,ds \lesssim (1+t)^{-\delta_\BQ},\\
& \int_{t-1}^t (t-s)^{-\frac{nd}2}\,(1+s)^{-\alpha\left(\delta_\BQ+\frac{\mu}2\right)+\frac{\mu}2}\,ds   \lesssim (1+t)^{-\delta_\BQ}.
\end{split} \]
In turn, we get
\[ \|(Nu-Nv)(t,\cdot)\|_{L^{\alpha+1}}\leq \int_0^t \| E(t,s, \cdot)\ast \Delta g(s,\cdot)\|_{L^{\alpha+1}} ds \leq C_2\,R^{\alpha-1} \|u-v\|_X\,(1+t)^{-\delta_\BQ-\frac{\mu}2}. \]
Recalling that $R=2C_1\|u_1\|_{\mathcal A}$ and letting $\|u_1\|_{\mathcal A}$ sufficiently small to obtain $C_2\,R^{\alpha-1}\leq 1/2$, we derive
\[ \|(Nu-Nv)\|_X \leq \frac12\,\|u-v\|_X. \]
Therefore, $F:X\to X$ and it is a contraction. By Banach's contraction principle, there is a unique $u\in X$ such that $Fu=u$, that is, $u$ is a solution to~\eqref{eq:NBQ}. Moreover, $\|u\|_X \leq R=2C_1\|u_1\|_{\mathcal A}$, that is, we proved~\eqref{eq:decayNIBQ}.
\end{proof}
%

\section*{Appendix}
\addcontentsline{toc}{section}{Appendix}
\renewcommand{\thesection}{A}
\setcounter{equation}{0}
\setcounter{proposition}{0}
\setcounter{lemma}{0}
In this section, we will recall some tools from special functions that have been applied in the treatments of the linear Cauchy problem \eqref{Eq:LinearProblem}.

\begin{lemma}[Bessel, Weber and Hankel functions, \cite{Bateman}] \label{hankel}
For the solutions of the Bessel equation of order $\gamma$ the following asymptotic estimates hold:
\begin{enumerate}[(i)]
\item The Bessel function $J_{\gamma}(\tau)$
\[ \Gamma_{\gamma}(\tau)=\tau^{-\gamma}J_{\gamma}(\tau),\]
is entire in $\gamma$ and $\tau$. In particular, for small arguments we have
\begin{equation*} \label{Besselest}
|J_{\gamma}(\tau)| \lesssim \tau^{\gamma}, \quad 0<\tau <1.
\end{equation*}
\item The Weber function $Y_{\gamma}(\tau)$ satisfies for every integer $n$
\[ Y_{n}(\tau)= \frac2{\pi}J_{n}(\tau)\ln\tau + A_n(\tau),\]
where $\tau^{n}A_{n}(\tau)$ is entire, non-null for $\tau=0$  and
\begin{equation*} \label{SecondBesselest}
|A_n(\tau)| \lesssim \tau^{-n}, \quad 0<\tau <1.
\end{equation*}
\item The Hankel functions $H^{\pm}_{\gamma}=J_{\gamma}\pm iY_{\gamma}$ satisfy
\[2(H^{\pm}_{\gamma})'(\tau)=H^{\pm}_{\gamma-1}(\tau)- H^{\pm}_{\gamma+1}(\tau) \quad  and \quad \tau (H^{\pm}_{\gamma})'(\tau)=\tau H^{\pm}_{\gamma-1}(\tau)-\gamma H^{\pm}_{\gamma}(\tau).\]
\item For large arguments $\tau\geq K>0$, $H^{\pm}_{\gamma}(\tau)$ can be written as
\begin{equation*}\label{HankelHigh}
H^{\pm}_{\gamma}(\tau)=e^{\pm i\tau} a_{\gamma}^{\pm}(\tau),
\end{equation*}
where $a_{\gamma}^{\pm}(\tau) \in S^{-\frac12}(K, \infty)$  is a classical symbol of order $-\frac12$.  More precisely, we have the following asymptotic behaviors \cite[Equations (10.2.5) and (10.2.6)]{Handbook=2010}:
\begin{align*}
H_{\gamma}^{+}(\tau) \sim \sqrt{\frac{2}{\pi \tau}}e^{i\left( \tau-\frac{1}{2}\gamma\pi-\frac{1}{4}\pi \right)} \qquad \text{and} \qquad H_{\gamma}^{-}(\tau) \sim \sqrt{\frac{2}{\pi \tau}}e^{-i\left(\tau-\frac{1}{2}\gamma\pi-\frac{1}{4}\pi \right)}.
\end{align*}
\item For small arguments $0<\tau\leq K <1$, we have
\begin{align*}\label{HankelLow}
|H^{\pm}_{\gamma}(\tau)|\lesssim
\begin{cases}
\tau^{-|\gamma|} & \text{if} \quad  \gamma\neq 0, \\
-\ln(\tau) & \text{if} \quad \gamma= 0.
\end{cases}
\end{align*}
\end{enumerate}
\end{lemma}

\section*{Acknowledgments}
Halit S. Aslan is supported by the S\~ao Paulo Research Foundation (FAPESP), Grant No.~2025/24251-0.
Marcelo R. Ebert is partially supported by the S\~ao Paulo Research Foundation (FAPESP), Grant No.~2024/12753-8.

\end{document}